\documentclass[11pt]{article}

\usepackage{amsfonts,amssymb}

\usepackage{amsthm}

\usepackage{amsmath}

\usepackage{amscd}
\usepackage{tabularx}
\usepackage{arydshln}
\usepackage{verbatim}
\usepackage{color}
\usepackage{a4wide}

\usepackage{tikz,tikz-cd} 
\usetikzlibrary{positioning,calc,decorations.pathreplacing}
\usepackage[all]{xy}

\theoremstyle{plain}

\usepackage{tikz-cd}
\tikzset{commutative diagrams/.cd}
\usetikzlibrary{intersections}
\usepackage{ mathrsfs }

\newtheorem{theorem}{Theorem}[section]
\newtheorem{lemma}[theorem]{Lemma}
\newtheorem{proposition}[theorem]{Proposition}

\newtheorem{Counter-example}[theorem]{Counter-example}
\newtheorem{remark}[theorem]{Remark}

\theoremstyle{definition}

\newtheorem{definition}[theorem]{Definition}

\theoremstyle{remark}

\usepackage{amssymb}

\def\R{\mathbb R}

\usetikzlibrary{positioning}
\usetikzlibrary{calc}

\title{\textbf{A Cartan geometry from a correspondence space in Riemannian geometry}}

\author{ 
Rodrigo Morón\footnote{Corresponding author.
}\\ 
 \texttt{Departamento de Matemáticas, Universidad de León, León, España} \\ 
 \texttt{$rmors@unileon.es$}\\
 \texttt{https://orcid.org/0000-0001-7651-5965}
 \and 
  Francisco J. Palomo\\ 
 \texttt{Departamento de Matemática Aplicada, Universidad de Málaga, Málaga, España} \\ 
 \texttt{$fpalomo@uma.es$}\\
  \texttt{https://orcid.org/0000-0002-1852-0667}
\date{\empty}
}

\begin{document}

\maketitle

\begin{abstract}
\noindent Cartan geometries are curved analogues of homogeneous spaces, and the correspondence space construction produces, from a Cartan geometry of one type, another of a different type over a larger base manifold. The unit tangent bundle of a Riemannian manifold is, in this language, a Cartan geometry of type $(\operatorname{Euc}(n),O(n-1))$ obtained in this way; we study Cartan geometries of this type in general, whether or not they arise as correspondence spaces. We show that every such geometry induces on its base manifold an almost contact metric structure, together with an orthogonal splitting of the tangent bundle and a compatible linear connection; conversely, these data determine the geometry, so that the correspondence is a bijection. When $n\ge3$, we single out a distinguished Cartan connection, which we call normal, by a condition on its torsion formulated in terms of the Spencer differential. This normalization is dictated by the correspondence space construction: the Cartan connection induced on the unit tangent bundle of a Riemannian manifold is normal.
\end{abstract}

\section{Introduction}
In Klein's Erlangen program \cite{Klein}, a geometry is described by a homogeneous space $G/H$, where $G$ is a Lie group acting as its group of symmetries and $H\subset G$ is a closed subgroup given by the stabilizer of a point. Cartan geometries \cite{CartanAffine,CartanHol} provide a unified description of a wide variety of geometric structures, obtained by taking such a homogeneous space as a model; in this sense, they can be seen as the curved analogues of homogeneous spaces, see Section~2. A Cartan geometry modeled on $G/H$ over a smooth manifold $M$ consists of an $H$-principal fiber bundle $\mathcal P\to M$ together with a $\mathfrak g$-valued $1$-form $\omega$ on it, called a Cartan connection, where $\mathfrak g$ denotes the Lie algebra of $G$. We refer to such a geometry as a Cartan geometry of type $(G,H)$ on $M$. The bundle $\mathcal P\to M$ is best understood when the model is first-order, meaning that $H$ acts faithfully on $\mathfrak g/\mathfrak h$: it is then a reduction of the frame bundle of $M$ to the structure group $H$. On the other hand, when the model is reductive, that is, $\mathfrak h$ admits an $H$-invariant complement $\mathfrak m$ in $\mathfrak g$, the Cartan connection splits as the sum of an $H$-principal connection and a distinguished $1$-form, the so-called soldering form~\eqref{21072026}. Reductions of the frame bundle and associated canonical connections have been studied for a long time \cite{Chern,Kobayashi,SV}, and reductive homogeneous spaces, together with their invariant connections, go back to \cite{Nomizu}. A general framework combining the first-order and reductive settings for Cartan geometries was developed in \cite{AM}. The classical example of a first-order reductive Cartan geometry is Riemannian geometry, regarded as a Cartan geometry modeled on Euclidean space $\mathbb E^n=\operatorname{Euc}(n)/O(n)$, see Section~\ref{22072026}.

Cartan geometries of one type can be used to produce Cartan geometries of another type, by means of the correspondence space construction. Given closed subgroups $H\subset P\subset G$ and a Cartan geometry $(\mathcal P\to M,\omega)$ of type $(G,P)$, the quotient $\mathcal C(M):=\mathcal P/H$ is a fiber bundle over $M$ with fiber $P/H$, and $\mathcal P$, equipped with the same $1$-form $\omega$, becomes a Cartan geometry of type $(G,H)$ on $\mathcal C(M)$. The curvature of such a geometry vanishes along the $\mathfrak p/\mathfrak h$ directions and, conversely, a geometry of type $(G,H)$ whose curvature satisfies this condition, together with a suitable integrability assumption, descends locally to one of type $(G,P)$. This converse allows one to start from an arbitrary Cartan geometry of type $(G,H)$ and to determine exactly when it arises from a Cartan geometry of type $(G,P)$. Correspondence spaces were introduced in \cite{Cap2005}, and we follow the treatment of \cite[Section~1.5.13]{CS09}.

Our starting point is a family of Cartan geometries, each obtained as the correspondence space of a Cartan geometry of type $(\operatorname{Euc}(n),O(n))$. To be precise, let $O(n-1)\subset O(n)$ be the stabilizer of a unit vector of $\mathbb R^n$, so that at the homogeneous level
\[
\operatorname{Euc}(n)/O(n)\cong\mathbb E^n,\qquad O(n)/O(n-1)\cong\mathbb S^{n-1},\qquad \operatorname{Euc}(n)/O(n-1)\cong\mathcal U\mathbb E^n,
\]
where $\mathcal U$ denotes the unit tangent bundle. Since the model $\operatorname{Euc}(n)/O(n)$ is first-order, the $O(n)$-principal fiber bundle of a Cartan geometry of type $(\operatorname{Euc}(n),O(n))$ on a manifold $M$ is a reduction of the frame bundle of $M$; its total space is therefore the orthonormal frame bundle $\mathcal F_g(M)$ of a Riemannian metric $g$ on $M$. The quotient $\mathcal F_g(M)/O(n-1)$ is the unit tangent bundle $\mathcal U M$, so that $\mathcal F_g(M)$, equipped with the same Cartan connection, becomes a Cartan geometry of type $(\operatorname{Euc}(n),O(n-1))$ on $\mathcal U M$. The geometry of these unit tangent bundles has been studied at length: their Sasaki metric, their contact metric structure, and their canonical distributions are all well documented; see \cite{Sasaki,YI73,Tashiro,Blair1}. A general Cartan geometry of type $(\operatorname{Euc}(n),O(n-1))$, however, need not arise from that construction. This raises the question:
\begin{quote}
what is a Cartan geometry of type $(\operatorname{Euc}(n),O(n-1))$, whether or not it arises as a correspondence space?
\end{quote}

The model $\operatorname{Euc}(n)/O(n-1)$ is again first-order and reductive, so that Cartan geometries of this type share the features described above. One may wonder why Cartan geometries should be used at all, since what is obtained is a first-order structure, for which a language of its own is already available. The reason is that a single $1$-form, the Cartan connection, encodes at once the first-order structure, through its soldering part, and a connection compatible with it, through its principal part. This viewpoint brings with it tools of clear geometric meaning, such as the curvature and the torsion of the Cartan connection. Moreover, the correspondence space construction, which is naturally formulated in this framework, motivates the class of geometries studied here and, as we shall see, guides the choice of a distinguished Cartan connection.

A question of the same kind has been considered in Lorentzian signature, where the counterpart of the unit tangent bundle is the observer space of a spacetime, the bundle of future-directed unit timelike vectors. Using Cartan geometry, Gielen and Wise \cite{GW} study the structure that a spacetime induces on its observer space, define an abstract observer space geometry to be a Cartan geometry modeled on the observer space of a homogeneous spacetime, and give conditions under which a spacetime can be recovered from such a geometry as a quotient. Their concern, however, is gravitational physics, rather than the two problems we take up in the Riemannian setting: to determine the underlying structure of a Cartan geometry of type $(\operatorname{Euc}(n),O(n-1))$, and to single out a distinguished Cartan connection.
 
Concerning the first, every Cartan geometry of type $(\operatorname{Euc}(n),O(n-1))$ induces on its $(2n-1)$-dimensional base manifold an almost contact metric structure, together with an orthogonal splitting of the tangent bundle and a compatible linear connection; conversely, these data determine the geometry, so that the correspondence is a bijection (Theorem~\ref{thm:euc-onminusone-equivalence}). Concerning the second, when $n\ge3$, a condition on the torsion of the Cartan connection singles out a preferred one, which we call normal (Theorem~\ref{thm:normal-adapted-connection}). The correspondence space construction is what motivates this normalization condition: given a torsion-free Cartan geometry of type $(\operatorname{Euc}(n),O(n))$, the Cartan connection of its correspondence space is normal (Proposition~\ref{prop:normal-compatible-with-riemannian-correspondence}).

The paper is organized as follows. Section~2 collects preliminary material on tangent and unit tangent bundles, and on Cartan geometries and correspondence spaces. Section~3 recalls how Riemannian geometry is described as a Cartan geometry of type $(\operatorname{Euc}(n),O(n))$.
 
Section~4 contains the main results. Subsection~\ref{subsec:euc-on-minus-one-model} analyzes the homogeneous model $\operatorname{Euc}(n)/O(n-1)$. Writing $\mathfrak g$ and $\mathfrak h$ for the Lie algebras of $\operatorname{Euc}(n)$ and $O(n-1)$, the quotient splits as
\[
\mathfrak g/\mathfrak h=\mathbb R e_0\oplus E\oplus B,
\]
where $\mathbb R e_0$ is the line spanned by a vector $e_0\in\mathfrak g/\mathfrak h$ and $E$ and $B$ are two copies of $\mathbb R^{n-1}$. The quotient adjoint representation of $O(n-1)$ fixes $e_0$ and acts by the standard representation on $E$ and on $B$; it follows that the model is first-order, and that $\mathfrak h$ admits an $O(n-1)$-invariant complement $\mathfrak m$ identified with $\mathfrak g/\mathfrak h$, so that the model is reductive. On $\mathfrak m$ we introduce an $O(n-1)$-equivariant endomorphism $\psi$ with
\[
\psi(e_0)=0,\qquad \psi(E)=B,\qquad \left.\psi^2\right|_{E\oplus B}=-\operatorname{Id},
\]
together with a natural $O(n-1)$-invariant Euclidean inner product.

Subsection~\ref{subsec:structure-induced} shows that every Cartan geometry of type $(\operatorname{Euc}(n),O(n-1))$ induces on its $(2n-1)$-dimensional base manifold $N$ a Riemannian metric $\bar g$, a unit vector field $U$ with dual $1$-form $\alpha$, an orthogonal splitting
\[
TN=\mathbb R U\oplus\mathcal E\oplus\mathcal B,\qquad \operatorname{rank}\mathcal E=\operatorname{rank}\mathcal B=n-1,
\]
and an endomorphism field $\Psi$ with $\Psi(\mathcal E)=\mathcal B$ satisfying
\[
\Psi^2=-\operatorname{Id}+\alpha\otimes U,\qquad \bar g(\Psi (V),\Psi (W))=\bar g(V,W)-\alpha(V)\alpha(W)
\]
(Proposition~\ref{thm:euc-onminusone-structure-new}); thus $(\bar g,U,\alpha,\Psi)$ is an almost contact metric structure. By reductivity, the $\mathfrak h$-part $\gamma$ of the Cartan connection is a principal connection, which induces a linear connection $\nabla^\gamma$ on $TN$ preserving $\bar g$, $U$, $\mathcal E$, $\mathcal B$ and $\Psi$ (Proposition~\ref{prop:canonical-linear-connection-induced}). On the correspondence space of a torsion-free Cartan geometry of type $(\operatorname{Euc}(n),O(n))$, Proposition~\ref{prop:unit-bundle-structures-agree-new} identifies these objects with the classical geometry of the unit tangent bundle described in Subsections~\ref{26072025} and~\ref{31052026}: $\bar g$ is the Sasaki metric, $U$ the geodesic vector field, $\alpha$ the contact form, $\mathbb R U\oplus\mathcal E$ and $\mathcal B$ the horizontal and vertical distributions, and $\Psi$ the tensor $\Phi$ determined by the almost complex structure $J$ and the Liouville normal vector field $\mathbb A$. However, Remark~\ref{rmk:gamma-not-sasaki-levi-civita} notes that $\nabla^\gamma$ cannot be the Levi-Civita connection of the Sasaki metric.
  
Subsection~\ref{subsec:inverse-construction} proves the converse. From a Cartan datum $(\bar g,U,\alpha,\mathcal E,\mathcal B,\Psi,\nabla)$ (Definition~\ref{def:adapted-cartan-datum}) one constructs a Cartan geometry of type $(\operatorname{Euc}(n),O(n-1))$. Theorem~\ref{thm:euc-onminusone-equivalence} shows this construction to be inverse to the one of Subsection~\ref{subsec:structure-induced}:
\begin{quote}
Cartan geometries of type $(\operatorname{Euc}(n),O(n-1))$ on a $(2n-1)$-manifold, up to isomorphism, correspond bijectively to Cartan datums.
\end{quote}
To conclude the subsection, Remark~\ref{rmk:recognition-correspondence-euc} gives the conditions under which a Cartan geometry of type $(\operatorname{Euc}(n),O(n-1))$ arises, locally, as a correspondence space of one of type $(\operatorname{Euc}(n),O(n))$.
  
Subsection~\ref{subsec:normal-adapted-cartan} singles out the normal Cartan connection, for $n\ge3$. The normalization procedure is controlled by the Spencer differential
\[
\partial\colon\mathfrak m^*\otimes\mathfrak h\to\Lambda^2\mathfrak m^*\otimes\mathfrak m,\qquad (\partial A)(X,Y)=[A(X),Y]-[A(Y),X],
\]
which is injective (Lemma~\ref{lem:normal-partial-injective}); equivalently, $\mathfrak h$ has trivial first prolongation. With the adjoint $\partial^*$ defined in \eqref{eq:normal-adjoint-definition}, a Cartan connection $\omega$ is called normal when its torsion $T^\omega$ satisfies $\partial^*T^\omega=0$ (Definition~\ref{def:normal-adapted-cartan}). Theorem~\ref{thm:normal-adapted-connection} then establishes:
\begin{quote}
underlying structures $(\bar g,U,\alpha,\mathcal E,\mathcal B,\Psi)$, that is, Cartan datums without their linear-connection part, correspond bijectively to Cartan geometries of type $(\operatorname{Euc}(n),O(n-1))$ with normal Cartan connection.
\end{quote}
This normalization is the one dictated by the correspondence space construction: the Cartan connection induced on the correspondence space of a torsion-free Cartan geometry of type $(\operatorname{Euc}(n),O(n))$ is normal (Proposition~\ref{prop:normal-compatible-with-riemannian-correspondence}). The linear connection determined by the normal Cartan connection is not computed here; its explicit form and geometric properties are left to future work (Remark~\ref{rmk:future-work}). Finally, Remark \ref{rmk:nilpotent-not-correspondence} presents examples of Cartan geometries of type $(\operatorname{Euc}(n),O(n-1))$ that do not arise as correspondence spaces.

\section{Preliminaries}

All manifolds are assumed to be smooth, Hausdorff, connected, and second countable. Throughout the article, we write $Tf$ for the differential of a smooth map $f$.

\subsection{Geometry of the tangent bundle}\label{26072025}

Let $(M,g)$ be a Riemannian manifold with $\dim M = n \geq 2$, Levi-Civita connection $\nabla$, and let $\pi \colon TM \to M$ denote the natural projection. The connector $c$ associated with $\nabla$ is defined by
\[
c : T(TM) \longrightarrow TM, 
\qquad 
\xi_{v} \longmapsto \left.\frac{\nabla \gamma}{dt}\right|_{t=0},
\]
where $\gamma$ is a curve in $TM$ such that $\gamma(0)=v\in T_{\pi(v)}M$ and $\dot\gamma(0)=\xi_v\in T_v(TM)$. Here, $\frac{\nabla \gamma}{dt}$ denotes the covariant derivative of the vector field $\gamma$ along the curve $\pi \circ \gamma$ in $M$. The map $c$ is well defined, and the pair $(c,\pi)$ defines a vector bundle morphism from $\pi_{TM} : T(TM) \to TM$ to $\pi : TM \to M$, where $\pi_{TM}$ denotes the natural projection. That is, the following diagram commutes:
\[
\begin{tikzcd}
T(TM) \arrow[r, "c"] \arrow[d, "\pi_{TM}"'] & TM \arrow[d, "\pi"] \\
T M \arrow[r, "\pi"'] &   M
\end{tikzcd}
\]
Recall that for every $v\in TM$, there is a natural identification
$$
(\,\,)_{v}\colon T_{\pi(v)}M\to T_{v}(T_{\pi(v)}M)\subset T_{v}(TM),\quad u\in T_{\pi(v)}M \mapsto (u)_v :=\left.\frac{d}{dt}\right|_{t=0}(v+tu).
$$
A straightforward computation shows that
$
c\big((u)_v\big) = u.
$

Using the connector $c$, we can lift the metric tensor $g$ to  $TM$ by defining
\[
\widehat{g}(\xi_v,\eta_v) 
= g\big(T_v\pi\cdot\xi_v, T_v\pi\cdot\eta_v\big) 
  + g\big(c(\xi_v), c(\eta_v)\big),
\quad \text{for all } \xi_v,\eta_v \in T_v(TM).
\]
The tensor $\widehat{g}$ is a Riemannian metric on $TM$, called the Sasaki metric induced by $g$;  see \cite[Section~4.5]{YI73} and \cite[Section~1.K]{Besse}.
The projection $\pi$ defines the vertical distribution on $TM$ by
$$
\mathcal{V}_v:=\mathrm{Ker}\,(T_{v}\pi)=T_{v}(T_{\pi(v)}M),
$$
and the Sasaki metric determines the horizontal distribution as its orthogonal complement,
$$
\mathcal{H}_v:=\mathcal{V}_v^{\perp}.
$$
Thus, we have the splitting $T(TM)=\mathcal{H}\oplus \mathcal{V} $, and the projection
$
\pi : (TM,\widehat{g}) \longrightarrow (M,g)$ is a Riemannian submersion. In particular, for every $v\in TM$,
$$
T_{v}\pi\colon \mathcal{H}_v\to T_{\pi(v)}M
$$
is an isometry. Note that $\mathcal{H}_v=\mathrm{Ker}\,(\left.c\right|_{T_{v}(TM)})$.

The connector $c$ also induces an almost complex structure $J$ on $TM$. Indeed, we have a vector bundle isomorphism
\begin{equation}\label{05062026ufhfud}
T(TM)\to TM \oplus TM, \quad \xi_v\mapsto (T_{v}\pi \cdot \xi_v , c(\xi_v))
\end{equation}
under which the almost complex structure $J\in \mathcal{T}_{(1,1)}(TM)$ corresponds to the map $(u,v)\mapsto (-v,u)$. A direct computation shows that
$$
\widehat{g}(J(\xi_v), J(\eta_v))=\widehat{g}(\xi_v, \eta_v).
$$

There are two distinguished vector fields on $TM$. The first is the Liouville vector field $\mathbb{A}\in \mathfrak{X}(TM)$, defined by $\mathbb{A}_v:=(v)_v$. This vector field is vertical and satisfies $\widehat{g}(\mathbb{A}_v, \mathbb{A}_v)=g(v,v)$ for every $v\in TM$. The second is the geodesic vector field $\mathbf{Z}_{g}\in \mathfrak{X}(TM)$, defined by
$$
\mathbf{Z}_{g}(v):= \left.\frac{d \dot\gamma_v}{dt}\right|_{t=0}, \quad v\in TM,
$$
where $\gamma_{v}\colon I \to M$ is the geodesic with initial conditions $\gamma_v(0)=\pi(v)$ and $\dot\gamma_v(0)=v$. Note that $\left.\frac{d \dot\gamma_v}{dt}\right|_{t=0}$ is the velocity of the curve $\dot\gamma_{v}\colon I \to TM$ at $t=0$. The vector field $\mathbf{Z}_g$ is horizontal and satisfies
$\widehat{g}(\mathbf{Z}_{g}(v), \mathbf{Z}_{g}(v))=g(v,v)$ for all $v\in TM$. Moreover, $J(\mathbb{A})=-\mathbf{Z}_{g}$.

\smallskip

Every vector field $V\in \mathfrak{X}(M)$ admits two natural lifts to $TM$. The horizontal lift $V^{\mathcal{H}}\in \mathfrak{X}(TM)$ is characterized by
$$
V^{\mathcal{H}}_v\in \mathcal{H}_v\quad \textrm{ and }\quad T_{v}\pi \cdot V^{\mathcal{H}}_v=V_{\pi(v)},
$$
for every $v\in TM$. The vertical lift $V^{\mathcal{V}}\in \mathfrak{X}(TM)$ is defined by
$$
V^{\mathcal{V}}_v:=\big(V_{\pi(v)}\big)_{v}.
$$
The Lie brackets of these lifted vector fields satisfy
$$
[V^{\mathcal{H}}, W^{\mathcal{H}}]=[V,W]^{\mathcal{H}}-\big(R(V,W)\big)^{\mathcal{V}},\quad [V^{\mathcal{H}}, W^{\mathcal{V}}]=\big(\nabla_{V}W\big)^{\mathcal{V}}, \quad [V^{\mathcal{V}}, W^{\mathcal{V}}]=0,
$$
for all $V,W\in \mathfrak{X}(M)$, where $R$ denotes the curvature tensor of $(M,g)$. The Levi-Civita connection $\widehat{\nabla}$ of $\widehat{g}$ is given by
$$
\widehat{\nabla}_{V^\mathcal{H}}W^\mathcal{H} = \big(\nabla_V W\big)^\mathcal{H} - \frac12 \big(R(V,W)\big)^\mathcal{V}, 
$$
$$\widehat{\nabla}_{V^\mathcal{H}}W^\mathcal{V} = \big(\nabla_V W\big)^\mathcal{V}+ \frac12 \big(R(\cdot,W)V\big)^\mathcal{H}, $$
$$\widehat{\nabla}_{V^\mathcal{V}}W^\mathcal{H} = \frac12 \big(R(\cdot,V)W\big)^\mathcal{H}, \quad \widehat{\nabla}_{V^\mathcal{V}}W^\mathcal{V} = 0.
$$
These formulas for the brackets of the lifts and the Levi-Civita connection are classical; see \cite[Section~4.5]{YI73} and \cite[Chapter~9]{Blair1}. Finally, the fibers of the Riemannian submersion $\pi : (TM, \widehat{g}) \to (M,g)$ are totally geodesic.

\subsection{Geometry of the unit tangent bundle}\label{31052026}

The unit tangent bundle of a Riemannian manifold $(M,g)$ is the subbundle of $TM$ defined by
$$
\mathcal{U}M=\{v\in TM:\ g(v,v)=1\}.
$$
That is, the restriction of the natural projection $\pi: TM \to M$ induces a fiber bundle $\pi:\mathcal{U}M \to M,$ whose standard fiber is the unit $(n-1)$-dimensional sphere $\mathbb{S}^{n-1}$ in the $n$-dimensional Euclidean space $\mathbb{E}^{n}$, defined by
$$
\mathbb{S}^{n-1}=\{x\in \mathbb{E}^{n}:\ \langle x,x\rangle=1\}.
$$

The Sasaki metric $\widehat{g}$  induces a Riemannian metric on the unit tangent bundle with unit normal vector field $\mathbb{A}$. Hence, we obtain
$$
T(\mathcal{U}M)=\{\xi\in T(TM): \widehat{g}(\xi, \mathbb{A})=0\}.
$$
In particular, the horizontal distribution $\mathcal{H}$ is tangent to $\mathcal{U}M$. Moreover, the geodesic vector field $\mathbf{Z}_g \in \mathfrak{X}(\mathcal{U}M)$ satisfies $\widehat{g}(\mathbf{Z}_{g}, \mathbf{Z}_{g})=1$ on $\mathcal{U}M.$ We also consider the vertical distribution
$$
V_v:=\mathrm{Ker}\,(T_v \pi),\quad v \in \mathcal{U}M,
$$
and we have the decomposition $T(\mathcal{U}M)= \mathcal{H}\oplus V$.  Finally, $\mathcal{U}M$ admits an almost contact metric structure
$(\widehat{g},\mathbf{Z}_g,\tau,\Phi)$ in the sense of
\cite[Chapter~4]{Blair1}, where
$$
\tau:=\widehat{g}(\mathbf{Z}_g,\cdot),
$$
and $\Phi$ is defined by the decomposition
\[
J(\xi)=\Phi(\xi)+\tau(\xi)\mathbb{A},
\qquad \xi\in T(\mathcal{U}M).
\]
Recall that, in general, an almost contact metric structure satisfies
\[
\Phi^{2}=-\operatorname{Id}+\tau\otimes\mathbf{Z}_g,
\qquad
\widehat{g}(\Phi(\xi),\Phi(\eta))
=
\widehat{g}(\xi,\eta)-\tau(\xi)\tau(\eta),
\qquad \xi,\eta\in T(\mathcal{U}M),
\]
where $\mathbf{Z}_g$ has unit length and $\tau$ is its
$\widehat{g}$-dual $1$-form. 

Moreover, on the unit tangent bundle endowed with the Sasaki metric, the
\(1\)-form \(\tau\) and the tensor \(\Phi\) defined above satisfy
\[
d\tau(\xi,\eta)
=
\widehat g\big(\xi,\Phi(\eta)\big),
\qquad
\xi,\eta\in T(\mathcal U M).
\]
This differential identity is not part of the definition of an almost contact
metric structure; it is a special feature of the unit tangent bundle with its
Sasaki metric. Precisely this identity, together with the almost contact metric
identities above, makes \((\widehat g,\mathbf Z_g,\tau,\Phi)\) a contact metric
structure. In particular, it forces \(d\tau\) to be non-degenerate on
\(\mathrm{Ker}(\tau)\), so that \(\tau\) is a contact form,
\[
\tau\wedge(d\tau)^{n-1}\neq0,
\]
with Reeb vector field \(\mathbf Z_g\); see, for instance, \cite[Chapter~9]{Blair1}.

\subsection{Cartan geometries and correspondence spaces}\label{31102024B}

For the sake of completeness, we recall several notions on Cartan geometries in a fully general context in this Section. We will focus only on the part of the theory relevant to this article and will closely follow  \cite[Section 1.5]{CS09}.

\subsubsection{Definition and properties}\label{130820295y}
 
For a principal fiber bundle $p:\mathcal{P}\to M$ with structure group $H$, let us denote by $r^{h}$ the right translation on $\mathcal{P}$ by $h \in H$, and by $\zeta_{X}\in \mathfrak{X}(\mathcal{P})$ the fundamental vector field corresponding to $X\in \mathfrak{h}=\mathrm{Lie}(H)$, given by
$
\zeta_{X}(u):=\left.\frac{d}{dt}\right|_{t=0}\left(u \cdot \mathrm{exp}\left(tX\right)\right).
$

Let $G$ be a Lie group and $H\subset G$ a closed subgroup, and let $\mathfrak{g}$ be the Lie algebra of $G$. A Cartan geometry of type $(G,H)$ on a manifold $M$ consists of:
\begin{enumerate}
\item An $H$-principal fiber bundle $p:\mathcal{P}\to M$.

\item A $\mathfrak{g}$-valued $1$-form $\omega \in \Omega^{1}(\mathcal{P}, \mathfrak{g})$, called the Cartan connection, such that for every $u\in \mathcal{P}, h\in H$, and $X\in \mathfrak{h}$, the following conditions hold:
\begin{enumerate}
\item $
\omega(u):T_{u}\mathcal{P}\to \mathfrak{g}
$
is a linear isomorphism.
\item  $(r^{h})^{*}(\omega)=\mathrm{Ad}(h^{-1})\circ \omega$.
\item $\omega(u)(\zeta_{X}(u))=X.$
 \end{enumerate} 
\end{enumerate}
For every $X\in \mathfrak{g}$, the constant vector field $\omega^{-1}(X)\in \mathfrak{X}(\mathcal{P})$ is defined by the condition $\omega(u) \left(\omega^{-1}\left(X\right)(u)\right)=X$ for all $u\in \mathcal{P}$. The curvature form of a Cartan geometry is given by $K:=d\omega+\frac{1}{2}[\omega, \omega]\in \Omega^{2}(\mathcal{P}, \mathfrak{g})$. Equivalently, we have
$$
K(\xi ,\eta)=d\omega (\xi, \eta)+[\omega(\xi), \omega (\eta)], \quad \xi, \eta\in \mathfrak{X}(\mathcal{P}). 
$$
A Cartan connection $\omega$ is said to be torsion-free when $K$
takes values in $\mathfrak{h}$.
All the information of $K$ is contained in the curvature function $\kappa : \mathcal{P}\to \Lambda^{2}\mathfrak{g}^{*}\otimes \mathfrak{g}$, defined as $$\kappa(u)(X,Y)=K(u)\left(\omega^{-1}(X)(u),\omega^{-1}(Y)(u)\right).$$ Since $K$ is horizontal, meaning that it vanishes when evaluated on a vertical tangent vector, the curvature function may be viewed as $\kappa : \mathcal{P}\to \Lambda^{2}(\mathfrak{g}/\mathfrak{h})^{*}\otimes \mathfrak{g}$, see \cite[Lemma 1.5.1]{CS09}.

The canonical projection $p:G\to G/H$, endowed with the (left) Maurer-Cartan form $\omega_{G}\in \Omega^{1}(G, \mathfrak{g})$,  is called the homogeneous model for Cartan geometries of type $(G,H)$. The Maurer-Cartan equation implies that the homogeneous model of any Cartan geometry has zero curvature \cite[Section 1.2.4]{CS09}.

A Cartan connection provides a description of the tangent bundle of the base manifold $M$ as the associated fiber bundle $\mathcal{P}\times_{{H}} (\mathfrak{g}/\mathfrak{h})$ for the quotient adjoint representation $\underline{\mathrm{Ad}}:H\longrightarrow \operatorname{GL}(\mathfrak{g}/\mathfrak{h)}$  given by
\begin{equation*}
\underline{\mathrm{Ad}}(h):\mathfrak{g}/\mathfrak{h} \to \mathfrak{g}/\mathfrak{h},\quad Y+\mathfrak{h}\mapsto \mathrm{Ad}(h)(Y)+\mathfrak{h},\end{equation*}
where $\mathrm{Ad}$ denotes the adjoint representation of $G$. Explicitly, for each $u\in \mathcal{P}$ with $p(u)=x\in M$, there is a canonical linear isomorphism 
$\phi_{u}:T_{x}M \to \mathfrak{g}/\mathfrak{h}$ such that the following diagram commutes:
\begin{equation}\label{isomor}
\begin{CD}
T_{u}\mathcal{P}  @>\omega(u)>> \mathfrak{g}\\
@V T_{u} p  VV @VV \mathrm{pr} V \\
T_{x}M @>>\phi_{u}  > \mathfrak{g}/\mathfrak{h}
\end{CD}\quad \quad \text{with } \phi_{uh}=\underline{\mathrm{Ad}}(h^{-1})\circ\phi_{u} \text{ for all }h\in H.
\end{equation}
Then, the canonical isomorphism of vector bundles over $M$ is given by
\begin{equation}\label{142}
TM \cong \mathcal{P}\times_{{H}} (\mathfrak{g}/\mathfrak{h}),\quad (x,v)\in TM\mapsto [u, \phi_{u}(v)],
\end{equation}
where $p(u)=x$, see \cite[Theorem 3.15]{Sharpe}.

The Cartan connection also determines a distinguished $\mathfrak{g}/\mathfrak{h}$-valued
$1$-form on $\mathcal{P}$, called the soldering form,
\begin{equation}\label{21072026}
\theta:=\mathrm{pr}\circ\omega\in\Omega^{1}(\mathcal{P},\mathfrak{g}/\mathfrak{h}).
\end{equation}
By the commutativity of the diagram \eqref{isomor}, the soldering form is
characterized by
$$
\theta(u)(\xi_u)=\phi_{u}\big(T_{u}p\cdot\xi_u\big),
\quad u\in\mathcal{P},\ \xi_u\in T_{u}\mathcal{P}.
$$
It is horizontal, $\theta(\zeta_{X})=0$ for every $X\in\mathfrak{h}$, and
$H$-equivariant,
$$
(r^{h})^{*}(\theta)=\underline{\mathrm{Ad}}(h^{-1})\circ\theta,\quad h\in H.
$$
When the model $G/H$ is reductive, that is, when $\mathfrak{g}$ carries an
$H$-invariant decomposition $\mathfrak{g}=\mathfrak{h}\oplus\mathfrak{m}$, the
Cartan connection splits as $\omega=\gamma+\theta$.
Its $\mathfrak{h}$-component $\gamma=\mathrm{pr}_{\mathfrak{h}}\circ\omega$
is a principal connection on $P$, while its $\mathfrak{m}$-component $\theta=\mathrm{pr}_{\mathfrak{m}}\circ\omega$ is the soldering form viewed through the natural isomorphism between $\mathfrak{m}$ and $\mathfrak{g}/\mathfrak{h}$, see \cite[Chapter 5]{Sharpe}.

An isomorphism of Cartan geometries $(p:\mathcal{P}\to M, \omega)$ and $(p':\mathcal{P}'\to M', \omega')$ with the same model $G/H$ is a principal fiber bundle isomorphism $(F,f)$ 
$$
 \begin{CD}
\mathcal{P}  @>F >> \mathcal{P}'\\
@V p VV @VV p' V \\
M @>>  f > M'
\end{CD}
$$
such that $F^{*}(\omega ')=\omega$.
That is, $F$ is a diffeomorphism from $\mathcal{P}$ to $\mathcal{P}'$ (and so is $f$) such that $F\circ r^{h}=r^{h}\circ F$ for all $h\in H$, and $F^{*}(\omega ')=\omega$. 
In the particular case that both Cartan geometries are the same, an isomorphism is called an automorphism.
The group $\mathrm{Aut}(\mathcal{P}, \omega)$ of all automorphisms of the Cartan geometry $(p:\mathcal{P}\to M, \omega)$ over a connected manifold $M$ is a Lie group (possibly with uncountably many connected
components) and has dimension at most $\mathrm{dim}(G)$, \cite[Theorem 1.5.11]{CS09}.

A Cartan geometry $p:\mathcal{P}\to M$ of type $(G,H)$ has curvature form $K=0$ if and only if every point $x\in M$ has an open neighborhood $U\subset M$ such that $\left(p:p^{-1}(U)\to U, \left.\omega\right|_{U}\right)$ is isomorphic to the restriction of the homogeneous model $\left(G\to G/H, \omega_{G}\right)$ to an open neighborhood of $o:=eH$,  \cite[Proposition 1.5.2]{CS09}.

\subsubsection{Correspondence spaces}\label{12062025}

This Subsection is based on \cite[Section 1.5.13]{CS09}. Let $(q:\mathcal{P}\to M, \omega)$ be a Cartan geometry of type $(G,P)$, and let $H\subset P$ be a closed subgroup. The correspondence space for $H\subset P$ is given by $\mathcal{C}(M):=\mathcal{P}/H$. 
The natural projection $f:\mathcal{C}(M)\to M$ defines a fiber bundle over $M$ with fiber given by the homogeneous space $P/H$,
and $(p\colon \mathcal{P}\to \mathcal{C}(M), \omega)$ is a Cartan geometry of type $(G,H)$ on $\mathcal{C}(M)$, see \cite[Proposition 1.5.13]{CS09}. The following commutative diagram holds:

$$
\begin{tikzpicture}
	
	\node (A) at (0,0) {$\mathcal{C}(M)$};
	\node (B) at (3,0) {$M$};
	\node (C) at (1.5,1.5) { $\mathcal{P}$};

	\draw[->] (A) -- node[midway, below] {$f$} (B);
	\draw[<-] (B) -- node[midway, right] {$q$} (C);
	\draw[->] (C) -- node[midway, left] {$p$} (A);
\end{tikzpicture}
$$
From (\ref{142}), we obtain a description of the tangent bundle of $\mathcal{C}(M)$ as  $\mathcal{P}\times _{{H}}(\mathfrak{g}/ \mathfrak{h})$. In these terms, the vertical distribution of $f$ corresponds to $\mathcal{P}\times _{{H}}(\mathfrak{p}/ \mathfrak{h})\subset \mathcal{P}\times _{{H}}(\mathfrak{g}/ \mathfrak{h})=T\mathcal{C}(M)$, \cite[Proposition 1.5.13]{CS09}.
The corresponding curvature function
$
k^{\mathcal{C}(M)}:\mathcal{P}\to \Lambda^{2}(\mathfrak{g}/\mathfrak{h})^{*}\otimes \mathfrak{g}
$
satisfies 
\begin{equation}\label{081224A}
  k^{\mathcal{C}(M)}(u)(X+\mathfrak{h}, \cdot)=0,   
\end{equation}
for every $X\in \mathfrak{p}$ and $u\in \mathcal{P}$.

On the other hand, the converse can also be characterized. Let $G$ be a Lie group and $H\subset P\subset G$ closed subgroups. Consider a Cartan geometry $(p:\mathcal{P}\to N, \omega)$ of type $(G,H)$ such that the distribution $\mathcal{V}(N):= \mathcal{P}\times _{H}(\mathfrak{p}/ \mathfrak{h}) \subset TN$ is integrable. 
A (local) twistor space for $N$ is a smooth manifold $M$ together with an open subset $U\subset N$ and a surjective submersion $f:U\to M$ such that $\mathcal{V}_{x}N=\mathrm{Ker}\,(T_{x}f)$, i.e., a (local) leaf space for the foliation defined by $\mathcal{V}(N),$ \cite[Definition 1.5.14]{CS09}. 

The curvature property (\ref{081224A}) locally characterizes the correspondence spaces as follows, \cite[Theorem 1.5.14]{CS09}.
Let $(p:\mathcal{P}\to N, \omega)$ be a Cartan geometry of type $(G,H)$ such that the distribution $\mathcal{V}(N)= \mathcal{P}\times_{H}(\mathfrak{p}/ \mathfrak{h}) \subset TN$ is integrable, and assume that the curvature function 
$
k^{N}:\mathcal{P}\to \Lambda^{2}(\mathfrak{g}/\mathfrak{h})^{*}\otimes \mathfrak{g}
$
satisfies $k^{N}(u)(X+\mathfrak{h}, \cdot)=0$ for every $X\in \mathfrak{p}$ and $u\in\mathcal{P}$. Then, for any sufficiently small local twistor space $f:U\to M$, there exists a Cartan geometry of type $(G,P)$ on $M$ such that the restricted Cartan geometry $\left(p:p ^{-1}(U) \to U, \left.\omega\right|_{U}\right)$ of type $(G,H)$ is isomorphic to an open subset of the correspondence space $\mathcal{C}(M)$. In the case where $P/H$ is connected, this Cartan geometry on $M$ is unique.

\section{Riemannian Geometry as a Cartan geometry}\label{22072026}
This Section is included for completeness. All the results recalled here are standard and follow Sharpe's treatment of Euclidean and Riemannian geometries \cite{Sharpe}.

Let $\mathbb{E}^{n}=(\mathbb{R}^{n},\langle\cdot,\cdot\rangle)$ be Euclidean space. We write the Euclidean group as the semidirect product
\[
\operatorname{Euc}(n)=\mathbb{R}^{n}\rtimes O(n)
=\left\{
\begin{pmatrix}
1&0\\
v&A
\end{pmatrix}: v\in\mathbb{R}^{n},\ A\in O(n)
\right\}.
\]
Its action on $\mathbb{R}^{n}$ is given by
\[
\begin{pmatrix}
1&0\\
v&A
\end{pmatrix}\cdot x=v+Ax,
\qquad x\in\mathbb{R}^{n}.
\]
The Lie algebra of $\operatorname{Euc}(n)$ is
\[
\mathfrak{euc}(n)=
\left\{
\begin{pmatrix}
0&0\\
v&B
\end{pmatrix}: v\in\mathbb{R}^{n},\ B\in\mathfrak{o}(n)
\right\}.
\]
We shall use the identification
\begin{equation}\label{eq:euc-quotient-on}
\mathfrak{euc}(n)/\mathfrak{o}(n)\cong\mathbb{R}^{n},
\qquad
\begin{pmatrix}
0&0\\
v&B
\end{pmatrix}+\mathfrak{o}(n)\longmapsto v.
\end{equation}

Recall that a Klein geometry $G/H$ is said to be of first-order if the quotient adjoint representation $\underline{\mathrm{Ad}}:H\longrightarrow \operatorname{GL}(\mathfrak{g}/\mathfrak{h)}$ is faithful.

\begin{proposition}\label{prop:euc-first-order-reductive}
The Klein geometry $\operatorname{Euc}(n)/O(n)$ is first-order and reductive. More precisely, under the identification \eqref{eq:euc-quotient-on},
\[
\underline{\mathrm{Ad}}(A)(v)=Av,
\qquad A\in O(n),\quad v\in\mathbb{R}^{n},
\]
and
\[
\mathfrak{euc}(n)=\mathfrak{o}(n)\oplus\mathfrak{t},
\qquad
\mathfrak{t}:=
\left\{
\begin{pmatrix}
0&0\\
v&0
\end{pmatrix}: v\in\mathbb{R}^{n}
\right\},
\]
with $\mathrm{Ad}(O(n))(\mathfrak{t})\subset\mathfrak{t}$.
\end{proposition}

Let $(q:\mathcal{P}\to M,\omega)$ be a Cartan geometry of type $(\operatorname{Euc}(n),O(n))$. For each $u\in\mathcal{P}$ with $q(u)=x\in M$, the Cartan connection defines, as in \eqref{isomor}, a linear isomorphism
\[
\phi_{u}:T_{x}M\longrightarrow \mathfrak{euc}(n)/\mathfrak{o}(n)\cong\mathbb{R}^{n},
\textrm{ with }
\phi_{uA}=\underline{\mathrm{Ad}}(A^{-1})\circ\phi_{u}.
\]
Thus the Euclidean inner product on $\mathbb{E}^{n}$ induces a Riemannian metric on $M$.

\begin{proposition}\label{prop:metric-from-euc-cartan}
A Cartan geometry $(q:\mathcal{P}\to M,\omega)$ of type $(\operatorname{Euc}(n),O(n))$ canonically determines a Riemannian metric $g$ on $M$ by
$$
g_{x}(v,w):=\left\langle \phi_{u}(v),\phi_{u}(w)\right\rangle,
\qquad v,w\in T_{x}M,
$$
where $u\in\mathcal{P}$ is any point with $q(u)=x$. Moreover, after fixing the canonical orthonormal basis $(E_{1},\ldots,E_{n})$ of $\mathbb{R}^{n}$, the map
\[
\mathcal{P}\longrightarrow \mathcal{F}_{g}(M),
\qquad
u\mapsto\big(\phi_{u}^{-1}(E_{1}),\ldots,\phi_{u}^{-1}(E_{n})\big),
\]
is an isomorphism of $O(n)$-principal fiber bundles over $M$.
Here the right action on $\mathcal F_g(M)$ is the usual one:
\[
(e_1,\ldots,e_n)\cdot A
=
\left(
\sum_{i=1}^{n} e_i A_{i1},
\ldots,
\sum_{i=1}^{n}e_i A_{in}
\right),
\qquad A=(A_{ij})\in O(n).
\]
\end{proposition}

Following Sharpe's terminology \cite[Page 234]{Sharpe}, a Euclidean geometry on $M$ is a Cartan geometry of type $(\operatorname{Euc}(n),O(n))$, and a Riemannian geometry is a torsion-free Euclidean geometry. Proposition \ref{prop:metric-from-euc-cartan} is the metric construction given in \cite[Proposition 3.2]{Sharpe}. Conversely, if $(M,g)$ is a Riemannian manifold, there is a unique Riemannian geometry inducing $g$, see \cite[Theorem 3.5]{Sharpe}. For this unique torsion-free geometry, the Cartan connection decomposes as
\[
\omega=\gamma^{LC}+\theta,
\]
where $\gamma^{LC}\in \Omega^1 (\mathcal{P}, \mathfrak{o}(n))$ is the Levi-Civita connection form and $\theta\in \Omega^1 (\mathcal{P}, \R^n)$ is the soldering form.

\section{A correspondence space in Riemannian geometry}
Let $(q:\mathcal F_g(M)\to M,\omega)$ be a Euclidean geometry on $M$, and let $g$ be the Riemannian metric induced by it. We apply the correspondence space construction of Subsection \ref{12062025} to the closed
subgroup
\[
O(n-1)\hookrightarrow O(n),
\qquad
h\longmapsto
\begin{pmatrix}
1&0\\
0&h
\end{pmatrix},
\]
that is, to the stabilizer of the first vector $E_{1}\in\mathbb{R}^{n}$. It is well known that
\[
O(n)/O(n-1)\cong\mathbb{S}^{n-1},
\qquad
A\,O(n-1)\longmapsto A E_{1}.
\]
Then, the correspondence space $\mathcal{C}(M)=\mathcal F_g(M)/O(n-1)$ is diffeomorphic to the unit tangent bundle $\mathcal{U}M=\{v\in TM:g(v,v)=1\}$, see \cite[Example 1.5.13]{CS09} for details. Under this diffeomorphism, the natural projection $f:\mathcal{C}(M)\to M$ is the
restriction $\pi:\mathcal{U}M\to M$ of the tangent bundle projection. Thus, the Cartan connection $\omega$ induces on the unit tangent bundle $\mathcal{U}M$ of $(M,g)$ a Cartan geometry of type $(\operatorname{Euc}(n),O(n-1)).$ This leads to the central question of the rest of the article: what, in general, is a Cartan geometry of type $(\operatorname{Euc}(n),O(n-1))$?

\subsection{The model $\operatorname{Euc}(n)/O(n-1)$}\label{subsec:euc-on-minus-one-model}

From now on we fix the following notation:
\[
G:=\operatorname{Euc}(n),\qquad P:=O(n),\qquad H:=O(n-1).
\]
At the level of Lie algebras we write
\[
\mathfrak g:=\mathfrak{euc}(n),\qquad
\mathfrak p:=\mathfrak{o}(n),\qquad
\mathfrak h:=\mathfrak{o}(n-1),\qquad \mathfrak t:=\mathbb R^n,
\]
so that
\[
\mathfrak g=\mathfrak p\oplus\mathfrak t.
\]
Write a vector in $\mathbb R^n$ as $(v_0, v)$, where
$v_0\in\mathbb R$ and $ v\in\mathbb R^{n-1}$. Every class in $\mathfrak g/\mathfrak h$ has a unique representative of the form
$$
(v_0, v,w):=
\begin{pmatrix}
0&0&0\\
v_0&0&-w^t\\
 v&w&0
\end{pmatrix},
\qquad
v_0\in\mathbb R,\quad  v,w\in\mathbb R^{n-1},
$$
where the lower right entry denotes the zero $(n-1)\times(n-1)$ matrix. We define
\begin{equation}\label{05062026}
e_0:=[(1,0,0)],\qquad
e_i:=[(0,E_i,0)],\qquad
b_i:=[(0,0,E_i)]
\end{equation}
where $(E_1,\ldots,E_{n-1})$ is the canonical basis of $\mathbb R^{n-1}$.
Thus
\begin{equation}\label{eq:W-decomposition-new}
\mathfrak g/\mathfrak h=\mathbb R e_0\oplus E\oplus B,
\qquad
E:=\operatorname{span}\{e_1,\ldots,e_{n-1}\},
\qquad
B:=\operatorname{span}\{b_1,\ldots,b_{n-1}\}
.
\end{equation}

We now describe explicitly the quotient adjoint representation $\underline{\mathrm{Ad}}:H\longrightarrow \operatorname{GL}(\mathfrak g/\mathfrak h).$  For each $h\in H$, regarded as an element of $G$, a direct matrix
multiplication gives
\begin{equation}\label{eq:ad-o-n-1-new}
\underline{\mathrm{Ad}}(h)\big([(v_0,v,w)]\big)=[(v_0,hv,hw)].
\end{equation}
It follows at once that the Klein geometry $G/H$ is first-order and reductive. Indeed, if
\(\underline{\mathrm{Ad}}(h)\) is the identity on \(\mathfrak g/\mathfrak h\), then \(h\)
acts trivially on the two copies of \(\mathbb R^{n-1}\), and hence \(h=\operatorname{Id}\). Moreover,
\[
\mathfrak m:=
\left\{
(v_0, v,w):v_0\in\mathbb R,\ v,w\in\mathbb R^{n-1}
\right\}
\]
is \(H\)-invariant by \eqref{eq:ad-o-n-1-new}, and
\begin{equation}\label{01062025}
\mathfrak g=\mathfrak h\oplus\mathfrak m.
\end{equation}
We shall identify \(\mathfrak m\) with \(\mathfrak g/\mathfrak h\) by the quotient map.

On the other hand, this Klein geometry carries a natural $H$-equivariant endomorphism of $\mathfrak m$, which will be important below. Namely, using the identification of $\mathfrak m$ with $\mathfrak g/\mathfrak h$, and slightly abusing notation by denoting again by $e_0,e_i$ and $b_i$ the corresponding elements of $\mathfrak m$, this endomorphism is the linear map
\begin{equation}\label{eq:model-psi-new}
\psi:\mathfrak m\longrightarrow \mathfrak m
\end{equation}
defined by
\[
\psi(e_0)=0,
\qquad
\psi(e_i)=b_i,
\qquad
\psi(b_i)=-e_i,
\qquad i=1,\ldots,n-1.
\]
Since $H$ acts in the same way on $E$ and on $B$, and fixes $e_0$, the map
$\psi$ commutes with $\underline{\mathrm{Ad}}(h)$ for every $h\in H$.

Finally, we put on $\mathfrak m$ the Euclidean inner product $\langle\cdot,\cdot\rangle_{\mathfrak m}$ for which $$(e_0,e_1,\ldots,e_{n-1},b_1,\ldots,b_{n-1})$$ is an orthonormal basis.  Formula \eqref{eq:ad-o-n-1-new} shows that this inner product is $H$-invariant.

\subsection{Structure induced by a Cartan geometry of type $(\operatorname{Euc}(n),O(n-1))$}\label{subsec:structure-induced}

We now show that any Cartan geometry with this model canonically induces on its base manifold the same type of almost contact metric structure that appears naturally on the unit tangent bundle of a Riemannian manifold.

\begin{proposition}\label{thm:euc-onminusone-structure-new}
Let $(p:\mathcal P\to N,\omega)$ be a Cartan geometry of type
$(G,H)$. Then $\dim N=2n-1$, and $N$ carries canonically the following objects:
\[
(\bar{g},U,\alpha,\mathcal E,\mathcal B,\Psi),
\]
where
\begin{enumerate}
\item\label{it:metric} $\bar{g}$ is a Riemannian metric;
\item $U\in\mathfrak X(N)$ is a unit vector field;
\item $\alpha:=\bar{g}(U,\cdot)$;
\item $\mathcal E$ and $\mathcal B$ are distributions such that $TN$ splits orthogonally as
$$
TN= \mathbb R U\oplus\mathcal E\oplus \mathcal B,
\qquad
\operatorname{rank}\mathcal E=\operatorname{rank}\mathcal B=n-1;
$$
\item\label{it:psi} $\Psi\in\mathcal T_{(1,1)}(N)$ satisfying
\[
\Psi^2=-\operatorname{Id}+\alpha\otimes U,
\qquad
 \bar{g}\big(\Psi(V),\Psi (W)\big)=\bar{g}(V,W)-\alpha(V)\alpha(W),
\]
where $V,W\in\mathfrak{X}(N)$, and
\[
\Psi(U)=0,
\qquad
\Psi(\mathcal E)=\mathcal B,
\qquad
\Psi(\mathcal B)=\mathcal E.
\]
\end{enumerate}
In particular, $(\bar{g},U,\alpha,\Psi)$ is an almost contact metric structure.
\end{proposition}
\begin{proof}
Let $u\in\mathcal P$ and put $x=p(u)$. As in \eqref{isomor}, the Cartan connection gives a
linear isomorphism
\[
\phi_u:T_xN\longrightarrow \mathfrak m.
\]
Since the inner product $\langle\cdot,\cdot\rangle_{\mathfrak m}$ is $H$-invariant, the formula 
\begin{equation}\label{05062026rfd}
\bar{g}_x(v,w):=\langle \phi_u(v),\phi_u(w)\rangle_{\mathfrak m},
\quad v,w\in T_xN,
\end{equation}
is independent of the chosen $u\in p^{-1}(x)$. Indeed, replacing $u$ by $uh$ gives
\[
\langle \phi_{uh}(v),\phi_{uh}(w)\rangle_{\mathfrak m}
=
\langle \underline{\mathrm{Ad}}(h^{-1})\phi_u(v),
\underline{\mathrm{Ad}}(h^{-1})\phi_u(w)\rangle_{\mathfrak m}
=
\langle\phi_u(v),\phi_u(w)\rangle_{\mathfrak m}.
\]
By the same argument as in \cite[Proposition 3.2]{Sharpe}, we conclude that $\bar{g}$ is differentiable. Hence, $\bar{g}$ defines a Riemannian metric on $N$.

For each $h \in H$, we have that $\underline{\mathrm{Ad}}(h)(e_0) = e_0$. Thus there exists $U \in \mathfrak{X}(N)$ with $U(x) \neq 0$ for all $x\in N$ given by
$$
U(x):=\phi_u^{-1}(e_0).
$$
Note that $U(x)$ is independent of $u\in p^{-1}(x)$. Since $e_0$ has length one in $\mathfrak m$, the vector field
$U$ has length one for $\bar{g}$. We set $\alpha:=\bar{g}(U,\cdot).$

The subspaces $E$ and $B$ in \eqref{eq:W-decomposition-new} are also $H$-invariant.
Thus the formulas
$$
\mathcal E_x:=\phi_u^{-1}(E),
\qquad
\mathcal B_x:=\phi_u^{-1}(B)
$$
define two smooth distributions on $N$, independent of the point $u\in p^{-1}(x)$.
The decomposition of $\mathfrak m$ in \eqref{eq:W-decomposition-new} is orthogonal; hence
\[
T_xN=\mathbb R U(x)\oplus\mathcal E_x\oplus \mathcal B_x
\]
is an orthogonal decomposition.

Finally, since the map $\psi:\mathfrak m\to \mathfrak m$ in \eqref{eq:model-psi-new} is $H$-equivariant, the
formula
\begin{equation}\label{050694836278}
\Psi_x(v):=\phi_u^{-1}\big(\psi(\phi_u(v))\big),
\qquad v\in T_xN,
\end{equation}
is independent of the chosen $u\in p^{-1}(x)$.

Smoothness of $\mathcal E,\mathcal B$
and $\Psi$ follows by writing the formulas above in a local section. The algebraic identities are checked on the basis of $\mathfrak m$. We have
\[
\psi^2(e_0)=0,
\qquad
\psi^2(e_i)=-e_i,
\qquad
\psi^2(b_i)=-b_i.
\]
Since $\alpha$ corresponds to the dual covector of $e_0$, this is exactly
\[
\Psi^2=-\operatorname{Id}+\alpha\otimes U.
\]
Moreover, let \(X,Y\in\mathfrak m\) be written in the orthonormal basis
\((e_0,e_1,\ldots,e_{n-1},b_1,\ldots,b_{n-1})\) as
\[
X=v_0e_0+\sum_{i=1}^{n-1} v_i e_i+\sum_{i=1}^{n-1} w_i b_i,
\qquad
Y=\bar{v}_0e_0+\sum_{i=1}^{n-1} \bar v_i e_i+\sum_{i=1}^{n-1} \bar w_i b_i.
\]
Then
\[
\langle \psi (X),\psi (Y)\rangle_{\mathfrak m}
=
\sum_{i=1}^{n-1} v_i\bar v_i+\sum_{i=1}^{n-1} w_i\bar w_i
=
\langle X,Y\rangle_{\mathfrak m}-v_0\bar v_0.
\]
Transporting this identity by $\phi_u^{-1}$ gives
\[
\bar{g}(\Psi (V),\Psi (W))=\bar{g}(V,W)-\alpha(V)\alpha(W).
\]
The relations $\Psi(U)=0$,$\Psi(\mathcal E)=\mathcal B$ and $\Psi(\mathcal B)=\mathcal E$ are immediate from the definition of $\psi$.
\end{proof}

\begin{remark}\label{rmk:almost-contact-not-contact}
{\rm The objects induced in Proposition~\ref{thm:euc-onminusone-structure-new} are algebraic in nature. More precisely, the vector space \(\mathfrak m\) carries tensors compatible with the adjoint representation and satisfying the same algebraic relations, and the Cartan connection transfers these tensors pointwise to \(N\). Thus the identities \[ \Psi^2=-\operatorname{Id}+\alpha\otimes U, \qquad \bar g(\Psi(V),\Psi(W)) = \bar g(V,W)-\alpha(V)\alpha(W) \] are pointwise, or zeroth-order, conditions. By contrast, the contact metric condition \[ d\alpha(V,W)=\bar g\big(V,\Psi(W)\big) \] is a first-order condition: it involves the exterior derivative of \(\alpha\), and therefore depends on the first derivatives of the \(1\)-form \(\alpha\), not only on its pointwise value. Hence this condition is not part of the underlying pointwise structure obtained in Proposition~\ref{thm:euc-onminusone-structure-new}.}
\end{remark}

\begin{remark}\label{19062026asd}
{\rm Since the model $G/H$ is first-order, any Cartan geometry $(p:\mathcal P\to N,\omega)$ of this type determines an $H$-structure on $N$. Equivalently, $\mathcal P$ can be identified, as an $H$-principal fiber bundle over $N$, with the bundle $\mathcal{F}_{\bar{g}}(N)$ of admissible orthonormal frames; see \cite[Chapter~5]{Sharpe} and \cite[Section~1.5]{CS09}. In the present notation, this bundle is described as follows. For $x\in N$,
\[
\left(\mathcal{F}_{\bar{g}}(N)\right)_x=\left\{\left(U(x),\phi_u^{-1}(e_1),\ldots,\phi_u^{-1}(e_{n-1}),\phi_u^{-1}(b_1),\ldots,\phi_u^{-1}(b_{n-1})\right):u\in\mathcal P,\ p(u)=x\right\}.
\]
Thus the admissible frames are precisely those induced by the Cartan connection from the fixed basis of $\mathfrak m$. In particular, all admissible frames at $x$ have the same first vector, namely $U(x)$. 

Under this identification the principal right action of $H$ corresponds to the change of admissible frame induced by the action $\underline{\mathrm{Ad}}(h)$ on $\mathfrak m$ given in \eqref{eq:ad-o-n-1-new}. Since $\underline{\mathrm{Ad}}(h)$ fixes $e_0$ and acts by the same matrix $h\in H$ on $E$ and on $B$, the right action of $H$ fixes $U(x)$ and rotates the $\mathcal E$-vectors $\phi_u^{-1}(e_i)$ and the $\mathcal B$-vectors $\phi_u^{-1}(b_i)$ by this same $h$; that is, $H$ acts as $\operatorname{diag}(1,h,h)\in O(2n-1)$. Consequently $\mathcal{F}_{\bar{g}}(N)$ is a reduction of the bundle of $\bar g$-orthonormal frames of $TN$ to the closed subgroup $H\hookrightarrow O(2n-1)$ acting in this way.
}
\end{remark}

We recall the standard way in which a principal connection on a principal fiber bundle induces linear connections on associated vector bundles \cite{KobayashiNomizu}. This construction will be applied below to the tangent bundle of $N$. Let $p:\mathcal P\to N$ be an $H$-principal fiber bundle, and let $\gamma\in\Omega^{1}(\mathcal P,\mathfrak h)$ be a principal connection on $\mathcal P$. Then, for every representation $\rho:H\to \mathrm{GL}(\mathbb V)$ of the Lie group $H$ on a vector space $\mathbb V$, the connection $\gamma$ induces a linear connection $\nabla^\gamma$ on the associated vector bundle $\mathcal P\times_H\mathbb V\to N$. Recall that $\nabla^{\gamma}$ is given as follows. Let $\sigma$ be a smooth section of $\mathcal{P}\times_{H}\mathbb{V}$, and let $U\subset N$ be an open subset. Then, for every $x\in U$, we have
\begin{equation}\label{13102023A}
	\nabla^\gamma_{W_{x}}\sigma:=\left[s(x),\,W_{x}(f)+\dot\rho\left(\gamma(s(x))(T_{x}s\cdot W_{x})\right)(f(x))\right],
\end{equation}
where $W\in \mathfrak{X}(N)$ and $\sigma|_U = [s,\, f]$ for a local section $s:U\subset N\rightarrow \mathcal{P}$ and a smooth function $f\in\mathcal{C}^{\infty}(U,\mathbb{V})$. We denote by $\dot\rho\colon \mathfrak{h}\to \mathfrak{gl}(\mathbb{V})$ the derivative of the representation $\rho.$

\begin{proposition}\label{prop:canonical-linear-connection-induced}
Let \((p:\mathcal P\to N,\omega)\) be a Cartan geometry of type \((G,H)\). Consider the reductive decomposition $\mathfrak g=\mathfrak h\oplus\mathfrak m$
described in \eqref{01062025}, and let
\[
\gamma:=\operatorname{pr}_{\mathfrak h}\circ\,\omega
\]
be the corresponding principal connection on \(\mathcal P\to N\). Then \(\gamma\) induces a linear connection \(\nabla^\gamma\) on \(TN\) such that
\[
\nabla^\gamma \bar{g}=0,
\qquad
\nabla^\gamma U=0,
\qquad
\nabla^\gamma\alpha=0,
\]
\[
\nabla^\gamma\mathcal E\subset\mathcal E,
\qquad
\nabla^\gamma\mathcal B\subset\mathcal B,
\qquad
\nabla^\gamma\Psi=0.
\]
\end{proposition}

\begin{proof}
By $(\ref{142})$, the Cartan connection identifies \(TN\) with the associated bundle $TN\cong\mathcal P\times_H\mathfrak m.$ Equivalently, if \(V\in\mathfrak X(N)\), then \(V\) corresponds to the unique
\(\underline{\mathrm{Ad}}\)-equivariant map
$
f_V:\mathcal P\longrightarrow\mathfrak m
$
defined by
\[
f_V(u):=\phi_u(V_{p(u)}).
\]
Thus $f_V(uh)=\underline{\mathrm{Ad}}(h^{-1})(f_V(u))$, for every $h\in H$.

We now take $W\in\mathfrak X(N)$, $x\in N$, and $u\in p^{-1}(x)$. Denote by $ W^{\gamma}_u\in T_u\mathcal P$ the $\gamma$-horizontal lift of $W_x$. Specializing $\eqref{13102023A}$ to the representation $\rho=\underline{\mathrm{Ad}}$, and transporting the connection induced by $\gamma$ on $\mathcal P\times_H\mathfrak m$ to $TN$ via the identification described above, the resulting linear connection $\nabla^\gamma$ on $TN$ satisfies
\[
\phi_u\big(\nabla^\gamma_{W_x}{} V\big)
=
T_uf_V\cdot W^{\gamma}_u.
\]

It remains to prove that \(\nabla^\gamma\) preserves the objects constructed in
Proposition \ref{thm:euc-onminusone-structure-new}. First, for vector fields \(V,W\in\mathfrak{X}(N)\), we have that
$
\bar{g}(V,W)\circ p
=
\langle f_V,f_W\rangle_{\mathfrak m}.
$
Differentiating along the \(\gamma\)-horizontal lift of \(X\in\mathfrak{X}(N)\) gives
\[
X(\bar{g}(V,W))
=
\bar{g}(\nabla^\gamma_XV,W)+\bar{g}(V,\nabla^\gamma_XW).
\]
Thus $\nabla^\gamma \bar{g}=0.$

The vector field \(U\) corresponds to the constant function $f_U(u)=e_0,$ for every $u\in\mathcal{P}$. Hence
$Tf_U\cdot W^{\gamma}=0$ for every \(W\in\mathfrak{X}(N)\), and therefore $\nabla^\gamma U=0.$ Since \(\alpha=\bar{g}(U,\cdot)\), the identities \(\nabla^\gamma \bar{g}=0\) and
\(\nabla^\gamma U=0\) imply $\nabla^\gamma\alpha=0.$

Next, a vector field \(V\in\mathfrak{X}(N)\) is a section of \(\mathcal E\) if and only if $f_V(\mathcal P)\subset E.$ Because \(E\subset\mathfrak m\) is a linear subspace, we also have $Tf_V\cdot W^{\gamma}\in E,$ for every \(W\in\mathfrak{X}(N)\). Therefore $\nabla^\gamma_WV\in\Gamma(\mathcal E),$ and hence $\nabla^\gamma\mathcal E\subset\mathcal E.$ The same argument applied to the subspace \(B\subset\mathfrak m\) gives $\nabla^\gamma\mathcal B\subset\mathcal B.$

Finally, the endomorphism $\psi:\mathfrak m\to\mathfrak m$ satisfies
$
f_{\Psi (V)}=\psi\circ f_V,
$
for every $V\in\mathfrak{X}(N)$. Since \(\psi\) is linear, $Tf_{\Psi (V)}\cdot W^{\gamma}
=
\psi\big(Tf_V\cdot W^{\gamma}\big).$ Transporting this identity back to \(TN\), we obtain $$\nabla^\gamma_W(\Psi (V))
=
\Psi(\nabla^\gamma_WV),$$ for all vector fields \(V,W\in\mathfrak{X}(N)\). Hence $\nabla^\gamma\Psi=0.$
\end{proof}

We now specialize the construction of Proposition~\ref{thm:euc-onminusone-structure-new} to the correspondence space \(\mathcal U M\cong \mathcal F_g(M)/H\), and compare the resulting objects with the standard geometry of the unit tangent bundle described in Subsection~\ref{31052026}.

\begin{proposition}\label{prop:unit-bundle-structures-agree-new}
Let $(M,g)$ be a Riemannian manifold, and let $(q:\mathcal F_g(M)\to M,\omega)$ be the Riemannian geometry associated with $(M,g)$. Under the identification $\mathcal U M \cong \mathcal F_g(M)/H,$ the Cartan connection $\omega$ determines a Cartan geometry of type $(G,H)$ on $\mathcal U M$. Then the objects induced on $\mathcal U M$ by this Cartan geometry are exactly as follows:
\begin{enumerate}
\item the metric $\bar{g}$  is the restriction of the Sasaki metric $\widehat g$ to $\mathcal U M$;
\item the vector field $U$ is the geodesic vector field $\mathbf Z_g$;
\item the $1$-form $\alpha$ is the $1$-form $\tau$;
\item $\mathbb R U\oplus\mathcal E$ is the horizontal distribution $\mathcal{H}$ of $\pi:\mathcal U M\to M$;
\item $\mathcal B$ is the vertical distribution $V$ of
$\pi:\mathcal U M\to M$;
\item the tensor $\Psi$ is the tensor $\Phi$ defined by
\[
J(\xi)=\Phi(\xi)+\widehat g(\mathbf Z_g,\xi)\mathbb A,
\qquad
\xi\in T(\mathcal U M),
\]
where $J$ is the almost complex structure on $TM$ and $\mathbb A$ is the Liouville normal vector field along $\mathcal U M\subset TM$.
\end{enumerate}
Consequently, $(\bar g,U,\alpha,\Psi)$ is the contact metric structure on $\mathcal U M$ induced by the Sasaki metric, as described in
Subsection~\ref{31052026}.
\end{proposition}

\begin{proof}
Recall that the Cartan connection on $\mathcal F_g(M)$ is
$\omega=\gamma^{LC}+\theta$, with $\gamma^{LC}$ the Levi-Civita connection form (valued in $\mathfrak p=\mathfrak o(n)$) and $\theta$ the soldering form (valued in $\mathfrak t=\mathbb R^n$).

Take $v\in\mathcal U M$, put $x=\pi(v)$, and choose a $g_x$-orthonormal basis
$\nu=(u_0,\dots,u_{n-1})\in\big(\mathcal F_g(M)\big)_x$ with $u_0=v$; its class modulo $H$ is
the point $v$. 
Let $\xi_v\in T_v(\mathcal U M)$, and choose a curve $v(t)\in\mathcal U M$ with $v(0)=v$ and
$\dot v(0)=\xi_v$, extended to a smooth $g$-orthonormal frame $\nu(t)=(u_0(t),\dots,u_{n-1}(t))$ with
$\nu(0)=\nu$ and $u_0(t)=v(t)$. Then $\phi_\nu(\xi_v)$ is the class of $\omega(\dot\nu(0))$ in
$\mathfrak m$, since $\dot\nu(0)$ projects to $\xi_v$ under $\mathcal F_g(M)\to\mathcal U M$. Here $\omega$ is the same Cartan connection as for $(M,g)$;
only the quotient changes, modulo $\mathfrak h$ instead of modulo $\mathfrak p$.

The $\mathfrak t$-part of $\phi_\nu(\xi_v)$ is $\theta(\nu)(\dot\nu(0))=\nu^{-1}(\dot x(0))$, where $x(t)=\pi(v(t))$, so there are unique scalars
$v_0,v_1,\dots,v_{n-1}$ with
\begin{equation}\label{eq:Tpi}
T_v\pi\cdot\xi_v=\dot x(0)=v_0u_0+\sum_{i=1}^{n-1}v_iu_i.
\end{equation}
Its $\mathfrak p$-part, modulo $\mathfrak h$, is the connector
$c(\xi_v)=\left.\frac{\nabla v}{dt}\right|_{t=0}$. Since $u_0(t)$ is unit, $c(\xi_v)$ is orthogonal to $u_0$, so there are unique scalars $w_1,\dots,w_{n-1}$ with
\begin{equation}\label{eq:connector}
c(\xi_v)=\sum_{i=1}^{n-1}w_iu_i.
\end{equation}
This gives
\begin{equation}\label{eq:phi-coords}
\phi_\nu(\xi_v)=v_0e_0+\sum_{i=1}^{n-1}v_ie_i+\sum_{i=1}^{n-1}w_ib_i,
\end{equation}
where $(e_0,e_1,\dots,e_{n-1},b_1,\dots,b_{n-1})$ is the basis of $\mathfrak m$ introduced in $(\ref{05062026})$.

As a consequence of Definition~\eqref{05062026rfd} and Equation~\eqref{eq:phi-coords}, for
\(\xi_v,\bar{\xi}_v\in T_v(\mathcal U M)\), we have
\[
\bar g(\xi_v,\bar{\xi}_v)
=
v_0\bar v_0+\sum_{i=1}^{n-1} v_i\bar v_i+\sum_{i=1}^{n-1} w_i\bar w_i.
\]
On the other hand, by the definition of the Sasaki metric and by
\eqref{eq:Tpi} and \eqref{eq:connector},
\[
\widehat g(\xi_v,\bar{\xi}_v)
=
g(T_v\pi\cdot\xi_v,T_v\pi\cdot\bar{\xi}_v)
+
g(c(\xi_v),c(\bar{\xi}_v))
=
v_0\bar v_0+\sum_{i=1}^{n-1} v_i\bar v_i+\sum_{i=1}^{n-1} w_i\bar w_i.
\]
Hence, \(\bar g=\widehat g|_{\mathcal U M}\).

By definition, \(U(v)=\phi_\nu^{-1}(e_0)\); in coordinates, this is the vector with
\(v_0=1\), \(v_i=0\), and \(w_i=0\). Equations~\eqref{eq:Tpi} and
\eqref{eq:connector} then give 
\[
T_v\pi\cdot U(v)=v,
\qquad
c(U(v))=0,
\]
which
characterizes the geodesic vector field $\mathbf{Z}_g(v)$. Hence
\(U=\mathbf Z_g\), and
\[
\alpha=\bar g(U,\cdot)=\widehat g(\mathbf Z_g,\cdot)=\tau.
\]

From \eqref{eq:Tpi}, \(\phi_\nu(\xi_v)\in B\), that is, \(v_0=0\) and
\(v_i=0\), precisely when \(T_v\pi\cdot\xi_v=0\). Thus \(\mathcal B=V\).
Similarly, \eqref{eq:connector} shows that
\(\phi_\nu(\xi_v)\in\mathbb R e_0\oplus E\), that is, \(w_i=0\), precisely
when \(c(\xi_v)=0\). Hence $\mathbb R U\oplus\mathcal E=\mathcal H$.

Under the connector identification \(T(TM)\cong TM\oplus TM\) introduced in $(\ref{05062026ufhfud})$, the Sasaki almost
complex structure is given by \(J(x,y)=(-y,x)\). Applying this to the pair
\((T_v\pi\cdot\xi_v,\,c(\xi_v))\), and using \eqref{eq:Tpi} and
\eqref{eq:connector}, the vector $J_v(\xi_v)$ corresponds to
\[
\big(-c(\xi_v),\,T_v\pi\cdot\xi_v\big)
=\left(-\sum_{i=1}^{n-1}w_iu_i,\, v_0u_0+\sum_{i=1}^{n-1}v_iu_i\right).
\]
Recall that the Liouville vector field \(\mathbb A_v\) corresponds to \((0,v)\),
whereas, from Definition $(\ref{050694836278})$, \(\Psi_v(\xi_v)\) corresponds to $$\left(-\sum_{i=1}^{n-1} w_iu_i,\sum_{i=1}^{n-1} v_iu_i\right).$$ Since $v_0=\widehat g_v(\mathbf Z_g(v),\xi_v)$, we obtain
\[
J_v(\xi_v)=\Psi_v(\xi_v)+\widehat g_v(\mathbf Z_g(v),\xi_v)\mathbb A_v,
\]
and therefore $\Psi=\Phi$.
\end{proof}

\begin{remark}\label{rmk:gamma-not-sasaki-levi-civita}
{\rm
In the correspondence space considered in Proposition~\ref{prop:unit-bundle-structures-agree-new}, the linear connection \(\nabla^\gamma\) of Proposition~\ref{prop:canonical-linear-connection-induced} is not the Levi-Civita connection \(\widehat\nabla\) of the Sasaki metric
\(\bar g=\widehat g|_{\mathcal U M}\). Indeed, by
Proposition~\ref{prop:canonical-linear-connection-induced} one has $\nabla^\gamma\alpha=0.$ By contrast, if \(\alpha\) were parallel with respect to \(\widehat\nabla\), then the torsion-freeness of \(\widehat\nabla\) would imply that \(d\alpha=0\). This is impossible in the present case, since \(\alpha=\tau\), where \(\tau\) is the contact form on
\(\mathcal U M\) described in Subsection~\ref{31052026}. Hence
\(\widehat\nabla\alpha\neq0\), and consequently \(\nabla^\gamma\neq\widehat\nabla\). In particular, the connection
\(\nabla^\gamma\) cannot be torsion-free.
}
\end{remark}

\subsection{The inverse construction and the equivalence}\label{subsec:inverse-construction}
We now turn to the converse problem: to describe Cartan geometries of type
\((G,H)\) in terms of data on \(N\). With respect to the reductive
decomposition \(\mathfrak g=\mathfrak h\oplus\mathfrak m\), a Cartan connection of this
type splits as
\[
\omega=\gamma+\vartheta,\qquad
\gamma:=\operatorname{pr}_{\mathfrak h}\circ\,\omega,\qquad
\vartheta:=\operatorname{pr}_{\mathfrak m}\circ\,\omega,
\]
where \(\gamma\) is an \(H\)-principal connection and \(\vartheta\) is the soldering
form. Proposition~\ref{thm:euc-onminusone-structure-new} shows that the
\(\mathfrak m\)-part \(\vartheta\) gives rise to the underlying structure $(\bar g,U,\alpha,\mathcal E,\mathcal B,\Psi).$
For \(n=2\), the algebra \(\mathfrak h=\mathfrak o(1)\) is zero, and hence there is no
principal-connection part in the decomposition above.  However, for \(n\geq3\), \(\mathfrak h=\mathfrak o(n-1)\neq0\) and the \(\mathfrak h\)-part \(\gamma\) has to be
accounted for. By Proposition~\ref{prop:canonical-linear-connection-induced}, the connection
\(\gamma\) induces a linear connection \(\nabla^\gamma\) on \(TN\) preserving the
underlying structure. Thus the natural data to retain on \(N\) consist of the underlying structure
together with the linear connection preserving it. This motivates the following definition; the fact that these data are exactly equivalent to Cartan geometries of
type \((G,H)\) is the content of Theorem~\ref{thm:euc-onminusone-equivalence}.

\begin{definition}\label{def:adapted-cartan-datum}
Let $N$ be a smooth manifold of dimension $2n-1$. A Cartan datum on $N$ is a tuple
\[
(\bar g,U,\alpha,\mathcal E,\mathcal B,\Psi,\nabla)
\]
such that $(\bar g,U,\alpha,\mathcal E,\mathcal B,\Psi)$ satisfies properties \textit{\ref{it:metric}}--\textit{\ref{it:psi}} of Proposition~\ref{thm:euc-onminusone-structure-new}, and $\nabla$ is a linear connection on
$TN$ satisfying
\[
\nabla \bar g=0,\qquad \nabla U=0,\qquad \nabla\alpha=0,
\]
\[
\nabla\mathcal E\subset\mathcal E,\qquad \nabla\mathcal B\subset\mathcal B,\qquad \nabla\Psi=0.
\]
\end{definition}

\begin{remark}\label{rmk:adapted-datum-redundancies}
{\rm Several of the conditions in Definition~\ref{def:adapted-cartan-datum} follow from the remaining ones and may be omitted without changing the notion of Cartan datum.
\begin{enumerate}
\item The $1$-form $\alpha=\bar g(U,\cdot)$ is determined by $\bar g$ and $U$.

\item The relation $\Psi(U)=0$ follows from $\Psi^{2}=-\operatorname{Id}+\alpha\otimes U$ and
$\alpha(U)=\bar g(U,U)=1$. Indeed, evaluating the identity at $U$ gives
$\Psi^{2}(U)=-U+\alpha(U)U=0$. Evaluating it instead at $\Psi(U)$ gives
$\Psi^{2}(\Psi(U))=-\Psi(U)+\alpha(\Psi(U))U$; the left-hand side equals
$\Psi(\Psi^{2}(U))=0$, so $\Psi(U)=\alpha(\Psi(U))U$. Applying $\Psi$ to this last identity
yields $0=\alpha(\Psi(U))\Psi(U)=\alpha(\Psi(U))^{2}U$, and since $U\neq0$ we get
$\alpha(\Psi(U))=0$, i.e., $\Psi(U)=0$.

\item Since $\mathcal E\oplus\mathcal B=\operatorname{Ker}(\alpha)$ and $\Psi^{2}=-\operatorname{Id}$
on $\operatorname{Ker}(\alpha)$, the two relations $\Psi(\mathcal E)=\mathcal B$ and
$\Psi(\mathcal B)=\mathcal E$ are equivalent: if $\Psi(\mathcal E)=\mathcal B$ then
$\Psi(\mathcal B)=\Psi^{2}(\mathcal E)=\mathcal E$, and
symmetrically for the converse. Hence only one of them need be assumed.

\item For the connection, $\nabla\alpha=0$ follows from $\nabla\bar g=0$ and $\nabla U=0$.

\item Once $\nabla\Psi=0$ holds, the inclusions $\nabla\mathcal E\subset\mathcal E$ and
$\nabla\mathcal B\subset\mathcal B$ are equivalent. Indeed, if $V\in\Gamma(\mathcal B)$ then
$V=\Psi (W)$ with $W\in\Gamma(\mathcal E)$, and $\nabla_{X}V=\Psi(\nabla_{X}W)\in\Psi(\mathcal
E)=\mathcal B$ whenever $\nabla_{X}W\in\mathcal E$; the reverse implication is symmetric. Thus only
one of the two inclusions need be assumed.
\end{enumerate}

Discarding the redundant conditions listed above, Definition~\ref{def:adapted-cartan-datum}
takes the following economical form, the omitted conditions then being automatic: a datum
consists of a Riemannian metric $\bar g$, a unit vector field $U$ (so that $\alpha=\bar g(U,\cdot)$),
an orthogonal splitting $TN=\mathbb R U\oplus\mathcal E\oplus\mathcal B$ with
$\operatorname{rank}\mathcal E=\operatorname{rank}\mathcal B=n-1$, and an endomorphism $\Psi$ with
$\Psi^{2}=-\operatorname{Id}+\alpha\otimes U$,
$\bar g(\Psi(\cdot),\Psi(\cdot))=\bar g(\cdot,\cdot)-\alpha\otimes\alpha$ and
$\Psi(\mathcal E)=\mathcal B$, together with a linear connection $\nabla$ such that
$\nabla\bar g=0$, $\nabla U=0$, $\nabla\mathcal E\subset\mathcal E$ and $\nabla\Psi=0$.}
\end{remark}

\begin{theorem}\label{thm:euc-onminusone-equivalence}
Let $N$ be a smooth manifold of dimension $2n-1$. Cartan geometries of type
$(G,H)$ on $N$, up to isomorphism, are in natural bijection with Cartan datums on $N$. More precisely:
\begin{enumerate}
\item from a Cartan geometry $(p:\mathcal P\to N,\omega)$ of type
$(G,H)$ one obtains the underlying structure
$(\bar g,U,\alpha,\mathcal E,\mathcal B,\Psi)$ of
Proposition~\ref{thm:euc-onminusone-structure-new}, together with the linear connection
$\nabla^{\gamma}$ of Proposition~\ref{prop:canonical-linear-connection-induced};
\item conversely, from a Cartan datum
$\mathcal S:=(\bar g,U,\alpha,\mathcal E,\mathcal B,\Psi,\nabla)$ one constructs canonically a
Cartan geometry $(p_{\mathcal S}:\mathcal P_{\mathcal S}\to N,\omega_{\mathcal S})$ of type
$(G,H)$;
\item these two constructions are mutually inverse, up to canonical isomorphism of $H$-principal fiber bundles.
\end{enumerate}
\end{theorem}

\begin{proof}

We first construct the Cartan geometry associated with a Cartan datum $\mathcal S$. Motivated by the \(H\)-structure described in Remark \ref{19062026asd}, we define \(\mathcal P_{\mathcal S}\) fiberwise as follows. For \(x\in N\), let \((\mathcal P_{\mathcal S})_x\) be the set of all bases
\[
b=\big(U(x),\varepsilon_{1},\dots,\varepsilon_{n-1},k_{1},\dots,k_{n-1}\big),
\]
where $(\varepsilon_{1},\dots,\varepsilon_{n-1})$ is a $\bar g$-orthonormal basis of $\mathcal
E_{x}$ and $k_{i}:=\Psi_x(\varepsilon_{i}),$ for $i=1,\dots,n-1.$ Because $\Psi(\mathcal E)=\mathcal B$, the vectors $k_{i}$ belong to $\mathcal B_{x}$; and since
$\mathcal E_{x}\subset\mathrm{Ker}\,(\alpha_{x})$ and
$\bar g(\Psi (V),\Psi (W))=\bar g(V,W)-\alpha(V)\alpha(W)$, the family $(k_{1},\dots,k_{n-1})$ is a
$\bar g$-orthonormal basis of $\mathcal B_{x}$. Together with the unit vector $U(x)$, orthogonal
to $\mathcal E_{x}\oplus\mathcal B_{x}$, the tuple $b$ is thus a $\bar g$-orthonormal basis of
$T_{x}N$ adapted to the orthogonal splitting
$T_{x}N=\mathbb R U(x)\oplus\mathcal E_{x}\oplus\mathcal B_{x}$. 

The group $H$ acts on the right by
\[
bh=\big(U(x),\varepsilon'_{1},\dots,\varepsilon'_{n-1},k'_{1},\dots,k'_{n-1}\big),
\qquad
\varepsilon'_{i}=\sum_{j=1}^{n-1}\varepsilon_{j}h_{ji},
\quad
k'_{i}=\sum_{j=1}^{n-1}k_{j}h_{ji}.
\]
Since $h\in H$, $(\varepsilon'_{1},\dots,\varepsilon'_{n-1})$ is again a
$\bar g$-orthonormal basis of $\mathcal E_{x}$, and
$k'_{i}=\sum_{j}\Psi_x(\varepsilon_{j})h_{ji}=\Psi_x(\varepsilon'_{i})$, so $b h$
is again a basis of the prescribed form. The action is free and transitive on each fiber, since two such bases differ by the unique $h\in H$ carrying one orthonormal basis of $\mathcal E_{x}$ to
the other. Hence $\mathcal P_{\mathcal S}$ is the reduction of the bundle of
$\bar g$-orthonormal frames of $TN$ to the closed subgroup $H\hookrightarrow
O(2n-1)$ that fixes $U$ and acts by the same $h$ on $\mathcal E$ and on
$\mathcal B$; in particular,
$p_{\mathcal S}:\mathcal P_{\mathcal S}\to N$ is a smooth $H$-principal fiber bundle. 

Let $I_{b}:\mathfrak m\rightarrow T_{x}N$ be the linear isomorphism carrying the basis $(e_0,e_1,\ldots,e_{n-1},b_1,\ldots,b_{n-1})$ of $\mathfrak m$ onto the basis $b$. We define
\[
\vartheta_{\mathcal S}\in\Omega^{1}(\mathcal P_{\mathcal S},\mathfrak m),
\qquad
\vartheta_{\mathcal S}(b)(\xi_b):=I_{b}^{-1}\big(T_{b}p_{\mathcal S}\cdot\xi_b\big),
\qquad
\xi_b\in T_{b}\mathcal P_{\mathcal S}.
\]
A direct computation gives
$I_{bh}=I_{b}\circ\underline{\mathrm{Ad}}(h)$, whence
$I_{bh}^{-1}=\underline{\mathrm{Ad}}(h^{-1})\circ I_{b}^{-1}$. Together with
$p_{\mathcal S}\circ r^{h}=p_{\mathcal S}$, this yields the equivariance
\[
(r^{h})^{*}(\vartheta_{\mathcal S})=\underline{\mathrm{Ad}}(h^{-1})\circ\vartheta_{\mathcal S},\qquad h\in H.
\]
Moreover $\vartheta_{\mathcal S}$ is horizontal, since $T_{b}p_{\mathcal S}\cdot\xi_b=0$ for every vertical $\xi_b\in T_{b}\mathcal P_{\mathcal S}$. Up to this point in the construction only the data \((\bar g,U,\alpha,\mathcal E,\mathcal B,\Psi)\) have been used.

Let $\nabla$ be the linear connection of the Cartan datum $\mathcal S$. Recall that a linear connection on $TN$ is the same datum as a principal connection on the frame bundle of $TN$. Explicitly, fix a basis $f=(f_{1},\dots,f_{2n-1})$ of $T_{x}N$ and let $f(t)$ be a curve of bases such that $f(0)=f$, lying over a curve $x(t)$ in $N$ with $x(0)=x$, and set $\xi:=\dot f(0)$. The induced connection $1$-form $\varpi$ has value $\varpi(f)(\xi)=\big(\varpi_{ij}(f)(\xi)\big)\in\mathfrak{gl}(2n-1)$ determined by
\begin{equation}\label{eq:induced-connection-form}
\left.\frac{\nabla f_{j}}{dt}\right|_{t=0}=\sum_{i=1}^{2n-1} f_{i}\,\varpi_{ij}(f)(\xi).
\end{equation}
The horizontal curves, those with $\varpi(f(t))(\dot f(t))=0$ for all $t$, are exactly those whose vectors are $\nabla$-parallel along the base curve, i.e., $\frac{\nabla f_{j}}{dt}= 0$ for every $j$. Since $\nabla\bar g=0$, parallel transport is isometric, so $\varpi$ restricts to the bundle of $\bar g$-orthonormal frames of $TN$, on which $\varpi(f)(\xi)\in\mathfrak o(2n-1)$ is skew-symmetric. We show that the remaining conditions of the datum make $\varpi$ restrict further to an $H$-principal connection on $\mathcal P_{\mathcal S}$: namely $\nabla U=0$, $\nabla\mathcal E\subset\mathcal E$ and $\nabla\Psi=0$ ensure that parallel transport carries adapted bases to adapted bases.

To compute this restriction, let $\xi_b\in T_{b}\mathcal P_{\mathcal S}$ and choose a curve
\[
b(t)=\big(U(x(t)),\varepsilon_{1}(t),\dots,\varepsilon_{n-1}(t),k_{1}(t),\dots,k_{n-1}(t)\big)
\]
in $\mathcal P_{\mathcal S}$ with $b(0)=b$ and $\dot b(0)=\xi_b$, where $x(t)=p_{\mathcal S}(b(t))$. We evaluate \eqref{eq:induced-connection-form} on the three blocks of this frame. First, $\nabla U=0$ gives
\[
\left.\frac{\nabla U}{dt}\right|_{t=0}=0 .
\]
Next, since each $\varepsilon_j(t)$ lies in $\mathcal E_{x(t)}$ and
$\nabla\mathcal E\subset\mathcal E$, we have
$\left.\frac{\nabla\varepsilon_{j}}{dt}\right|_{t=0}\in\mathcal E_{x}$, and
differentiating $\bar g(\varepsilon_{i}(t),\varepsilon_{j}(t))=\delta_{ij}$ with
$\nabla\bar g=0$ shows that its matrix in the orthonormal basis
$(\varepsilon_{1},\dots,\varepsilon_{n-1})$ is skew-symmetric. Hence there are
unique scalars $\gamma_{ij}(b)(\xi_b)$ with
\begin{equation}\label{eq:gamma-block}
\left.\frac{\nabla\varepsilon_{j}}{dt}\right|_{t=0}
=\sum_{i=1}^{n-1}\varepsilon_{i}\,\gamma_{ij}(b)(\xi_b)\quad\text{and}
\quad
\gamma_{\mathcal S}(b)(\xi_b):=\big(\gamma_{ij}(b)(\xi_b)\big)\in\mathfrak h.
\end{equation}
Finally, differentiating $k_{j}(t)=\Psi_{x(t)}(\varepsilon_{j}(t))$, we obtain that
\[
\left.\frac{\nabla k_{j}}{dt}\right|_{t=0}
=\big(\nabla_{\dot x(0)}\Psi\big)(\varepsilon_{j})
+\Psi_{x}\!\left(\left.\frac{\nabla\varepsilon_{j}}{dt}\right|_{t=0}\right).
\]
The first term vanishes because $\nabla\Psi=0$; by the previous display and the linearity of $\Psi_{x}$, the second gives
\[
\left.\frac{\nabla k_{j}}{dt}\right|_{t=0}
=\Psi_{x}\!\left(\sum_{i=1}^{n-1}\varepsilon_{i}\,\gamma_{ij}(b)(\xi_{b})\right)
=\sum_{i=1}^{n-1}k_{i}\,\gamma_{ij}(b)(\xi_{b}),
\]
so the same matrix $\gamma_{\mathcal S}(b)(\xi_{b})$ governs the $\mathcal B$-block. Thus, in the adapted frame $b$, the matrix \eqref{eq:induced-connection-form} is
\[
\varpi(b)(\xi_{b})=
\begin{pmatrix}
0 & 0 & 0\\
0 & \gamma_{\mathcal S}(b)(\xi_{b}) & 0\\
0 & 0 & \gamma_{\mathcal S}(b)(\xi_{b})
\end{pmatrix},
\]
relative to the splitting $T_{x}N=\mathbb R\,U(x)\oplus\mathcal E_{x}\oplus\mathcal B_{x}$. This matrix lies in $\mathfrak h$, embedded in $\mathfrak o(2n-1)$ as the diagonal copy $A\mapsto\operatorname{diag}(0,A,A)$, which is the infinitesimal form of the embedding $h\mapsto\operatorname{diag}(1,h,h)$ of Remark~\ref{19062026asd}. Consequently, the restriction
\[
\gamma_{\mathcal S}:=\varpi|_{\mathcal P_{\mathcal S}}\in\Omega^{1}(\mathcal P_{\mathcal S},\mathfrak h)
\]
is an $\mathfrak h$-valued $1$-form on $\mathcal P_{\mathcal S}$.

That $\gamma_{\mathcal S}$ is a principal connection follows from \eqref{eq:gamma-block}. For $A\in\mathfrak h$, the fundamental vector field $\zeta_{A}(b)=\left.\frac{d}{dt}\right|_{t=0}\left(b \cdot \mathrm{exp}\left(tA\right)\right)$ is the velocity of the curve $b(t)=b\cdot\exp(tA)$, for which $x(t)\equiv x$ and $\varepsilon_{j}(t)=\sum_{i}\varepsilon_{i}\,(\exp tA)_{ij}$. The base curve is constant, so $\dot x(t)\equiv0$ and the covariant derivative along it reduces to the ordinary derivative; hence $\left.\frac{\nabla\varepsilon_{j}}{dt}\right|_{t=0}=\sum_{i}\varepsilon_{i}\,A_{ij}$, whence $\gamma_{\mathcal S}(b)(\zeta_{A}(b))=A$. The equivariance
\[
(r^{h})^{*}(\gamma_{\mathcal S})=\operatorname{Ad}(h^{-1})\circ\gamma_{\mathcal S},\qquad h\in H,
\]
is the transformation law of \eqref{eq:gamma-block} under the change of basis $b\mapsto bh$. Therefore $\gamma_{\mathcal S}\in\Omega^{1}(\mathcal P_{\mathcal S},\mathfrak h)$ is an $H$-principal connection on $\mathcal P_{\mathcal S}$.

Set $\omega_{\mathcal S}:=\gamma_{\mathcal S}+\vartheta_{\mathcal S}\in\Omega^{1}(\mathcal P_{\mathcal S},\mathfrak g)$,
according to $\mathfrak g=\mathfrak h\oplus\mathfrak m$. If
$\vartheta_{\mathcal S}=(\vartheta_{\mathcal S}^{0},\vartheta_{\mathcal S}^{E},\vartheta_{\mathcal S}^{B})$ denotes its $e_{0}$-, $E$- and $B$-components, then in matrix form
\[
\omega_{\mathcal S}=
\begin{pmatrix}
0&0&0\\
\vartheta_{\mathcal S}^{0}&0&-(\vartheta_{\mathcal S}^{B})^{t}\\
\vartheta_{\mathcal S}^{E}&\vartheta_{\mathcal S}^{B}&\gamma_{\mathcal S}
\end{pmatrix},
\]
consistent with the description of $\mathfrak g/\mathfrak h$ in
Subsection~\ref{subsec:euc-on-minus-one-model}. By construction, $\omega_{\mathcal S}$ is readily checked to be a Cartan connection of type $(G,H)$, so that $(p_{\mathcal S}:\mathcal P_{\mathcal S}\to N,\omega_{\mathcal S})$ is a Cartan geometry of that type.

We have associated with each Cartan datum a Cartan geometry of type $(G,H)$. To establish the asserted bijection, it remains to check that this assignment and the one extracting a Cartan datum from a Cartan geometry are mutually inverse; we verify the two composites in turn.

Let $(p:\mathcal P\to N,\omega)$ be a Cartan geometry of type $(G,H)$, and write
$\gamma=\operatorname{pr}_{\mathfrak h}\circ\,\omega$ and
$\vartheta=\operatorname{pr}_{\mathfrak m}\circ\,\omega$. Propositions
\ref{thm:euc-onminusone-structure-new} and \ref{prop:canonical-linear-connection-induced}
yield the Cartan datum $\mathcal S=(\bar g,U,\alpha,\mathcal E,\mathcal B,\Psi,\nabla^{\gamma})$; let
$(p_{\mathcal S}:\mathcal P_{\mathcal S}\to N,\omega_{\mathcal S})$ be the Cartan geometry of
type $(G,H)$ associated with $\mathcal S$ as above. By construction, $\mathcal P_{\mathcal S}$ is precisely the bundle $\mathcal F_{\bar g}(N)$ of
admissible orthonormal frames described in Remark~\ref{19062026asd}; accordingly, the map
\[
F:\mathcal P\longrightarrow\mathcal P_{\mathcal S},\qquad
F(u)=\big(U(x),\,\phi_{u}^{-1}(e_{1}),\dots,\phi_{u}^{-1}(e_{n-1}),\,
\phi_{u}^{-1}(b_{1}),\dots,\phi_{u}^{-1}(b_{n-1})\big),
\]
where $x=p(u)$ and $\phi_{u}$ is the isomorphism given in \eqref{isomor}, is an isomorphism of $H$-principal fiber bundles. Indeed, $F(u)\in(\mathcal P_{\mathcal S})_{x}$: $\phi_u$ is a linear isometry, so the vectors $\phi_{u}^{-1}(e_{i})$ form a
$\bar g$-orthonormal basis of $\mathcal E_{x}=\phi_u^{-1}(E)$, and $\psi(e_{i})=b_{i}$ together with
$\Psi_{x}=\phi_u^{-1}\circ\psi\circ\phi_u$ gives $\phi_{u}^{-1}(b_{i})=\Psi_x(\phi_{u}^{-1}(e_{i}))$. It remains to compare the Cartan connections. For $u\in\mathcal P$ the frame $F(u)$ satisfies
$I_{F(u)}=\phi_{u}^{-1}$, since both maps send $(e_{0},e_{i},b_{i})$ to
$(U(x),\phi_{u}^{-1}(e_{i}),\phi_{u}^{-1}(b_{i}))$. Using $p_{\mathcal S}\circ F=p$ and the definition of the soldering form,
\[
\big(F^{*}(\vartheta_{\mathcal S})\big)(u)(\xi_{u})
=I_{F(u)}^{-1}\big(T_{u}p\cdot\xi_{u}\big)
=\phi_{u}\big(T_{u}p\cdot\xi_{u}\big)
=\vartheta(u)(\xi_{u}).
\]
For the $\mathfrak h$-components, $F^{*}(\gamma_{\mathcal S})$ is again a principal connection, since $F$ is an isomorphism of $H$-principal fiber bundles. As any two principal connections agree on
fundamental vector fields, $F^{*}(\gamma_{\mathcal S})=\gamma$ will follow once they are shown to have the same horizontal distribution. Let $\xi_{u}\in T_{u}\mathcal P$ be $\gamma$-horizontal, and let $u(t)$ be a $\gamma$-horizontal curve with $u(0)=u$ and $\dot u(0)=\xi_{u}$; write $x(t)=p(u(t))$, so that $\dot u(t)=\big(\dot x(t)\big)^{\gamma}_{u(t)}$ for all $t$. Fix $X\in\mathfrak m$ and set $V(t):=\phi_{u(t)}^{-1}(X)\in T_{x(t)}N$, a vector field along $x(t)$. Its equivariant representative is constant along $u(t)$, \[ f_{V}\big(u(t)\big)=\phi_{u(t)}\big(V(t)\big)=\phi_{u(t)}\big(\phi_{u(t)}^{-1}(X)\big)=X . \] Applying the formula $\phi_{u}\big(\nabla^{\gamma}_{W_{x}}V\big)=T_{u}f_{V}\cdot W^{\gamma}_{u}$ of Proposition~\ref{prop:canonical-linear-connection-induced} along $u(t)$, with $W^{\gamma}_{u(t)}=\dot u(t)$, \[ \phi_{u(t)}\!\left(\frac{\nabla^{\gamma}V}{dt}\right) =T_{u(t)}f_{V}\cdot\dot u(t) =0 , \] and since $\phi_{u(t)}$ is an isomorphism, $\dfrac{\nabla^{\gamma}V}{dt}=0$: the field $\phi_{u(t)}^{-1}(X)$ is $\nabla^{\gamma}$-parallel along $x(t)$. Taking $X\in\{e_{0},e_{1},\dots,e_{n-1},b_{1},\dots,b_{n-1}\}$ and recalling $\phi_{u(t)}^{-1}(e_{0})=U(x(t))$, every vector of the frame \[ F\big(u(t)\big)=\big(U(x(t)),\,\phi_{u(t)}^{-1}(e_{1}),\dots,\phi_{u(t)}^{-1}(e_{n-1}),\, \phi_{u(t)}^{-1}(b_{1}),\dots,\phi_{u(t)}^{-1}(b_{n-1})\big) \] is $\nabla^{\gamma}$-parallel along $x(t)$. By the construction of $\gamma_{\mathcal S}$ in \eqref{eq:gamma-block}, $\gamma_{\mathcal S}(b)(\xi_{b})$ vanishes precisely when the frame with $\dot b(0)=\xi_{b}$ is $\nabla^{\gamma}$-parallel at $t=0$; as $\frac{d}{dt}\big|_{t=0}F(u(t))=T_{u}F\cdot\xi_{u}$, this gives \[ \gamma_{\mathcal S}(F(u))\big(T_{u}F\cdot\xi_{u}\big)=0 . \] Hence $F$ carries $\gamma$-horizontal vectors to $\gamma_{\mathcal S}$-horizontal vectors, and therefore $F^{*}(\gamma_{\mathcal S})=\gamma$. Adding the two identities,
\[
F^{*}(\omega_{\mathcal S})=F^{*}(\gamma_{\mathcal S})+F^{*}(\vartheta_{\mathcal S})
=\gamma+\vartheta=\omega,
\]
so the reconstructed Cartan geometry is canonically isomorphic to the original one.

Conversely, start from a Cartan datum $\mathcal S=(\bar g,U,\alpha,\mathcal E,\mathcal B,\Psi,\nabla)$
and let $(p_{\mathcal S}:\mathcal P_{\mathcal S}\to N,\omega_{\mathcal S})$ be the Cartan geometry
constructed above. For $b\in\mathcal P_{\mathcal S}$ with $x=p_{\mathcal S}(b)$, the linear
isomorphism $\phi_{b}$ satisfies $\phi_{b}=I_{b}^{-1}$. Write
$(\bar g',U',\alpha',\mathcal E',\mathcal B',\Psi')$ for the structure induced by
$\omega_{\mathcal S}$. Since $\phi_{b}=I_{b}^{-1}$ and $I_{b}$ is the linear isometry sending the
orthonormal basis $(e_{0},e_{i},b_{i})$ of $\mathfrak m$ to the
$\bar g$-orthonormal basis $b$, a direct computation gives $\bar g'=\bar g$, $U'=U$,
$\alpha'=\alpha$, $\mathcal E'=\mathcal E$, $\mathcal B'=\mathcal B$ and $\Psi'=\Psi$. The only point
to note is the value of $\Psi'$ on the $\mathcal B$-vectors: by \eqref{050694836278},
\[
\Psi'_{x}(k_{i})=I_{b}\big(\psi(b_{i})\big)=-\varepsilon_{i}=\Psi_{x}(k_{i}),
\]
where the last equality uses $\Psi_{x}(k_{i})=\Psi_{x}^{2}(\varepsilon_{i})=-\varepsilon_{i}$, valid
because $\varepsilon_{i}\in\mathrm{Ker}\,(\alpha_{x})$. Finally, it is straightforward to check that $\nabla$ and $\nabla^{\gamma_{\mathcal S}}$ have the same parallel frames, hence the same parallel transport, so $\nabla=\nabla^{\gamma_{\mathcal S}}$. Therefore the Cartan datum induced by
$(p_{\mathcal S}:\mathcal P_{\mathcal S}\to N,\omega_{\mathcal S})$ is $\mathcal S$. Together, the two reconstructions establish the asserted bijection.
\end{proof}
The correspondence established in
Theorem~\ref{thm:euc-onminusone-equivalence} can be summarised by the following diagram, where all arrows are to be
understood as correspondences:
\smallskip

\[
\begin{tikzpicture}[
  baseline=(current bounding box.center),
  >=stealth,
  semithick,
  every node/.style={inner sep=1pt}
]
% lado izquierdo
\node (lp) {$\text{Cartan datum }\bigl($};
\node[right=1pt of lp] (data) {$\bar g,U,\alpha,\mathcal E,\mathcal B,\Psi,$};
\node[right=1pt of data] (comma) {$ $};
\node[right=1pt of comma] (nabla) {$\nabla$};
\node[right=1pt of nabla] (rp) {$\bigr)$};

% lado derecho
\node[right=14mm of rp] (rpre) {$\bigl($};
\node[right=1pt of rpre] (bundle) {$p:\mathcal P\to N,$};
\node[right=1pt of bundle] (rcomma) {$ $};
\node[right=1pt of rcomma] (omegapre) {$\omega=$};
\node[right=1pt of omegapre] (gamma) {$\gamma$};
\node[right=1pt of gamma] (plus) {$+$};
\node[right=1pt of plus] (theta) {$\vartheta$};
\node[right=1pt of theta] (rpost)
  {$\bigr)\text{ of type }(G,H)$};

% doble flecha central
\node at ($(rp.east)!0.5!(rpre.west)$) {$\Longleftrightarrow$};

% llave izquierda
\draw[decorate,decoration={brace,mirror,amplitude=4pt}]
  ([yshift=-2pt]data.south west) -- ([yshift=-2pt]data.south east);

\coordinate (leftbrace) at
  ([yshift=-8pt]$(data.south west)!0.5!(data.south east)$);

% llave derecha bajo p:P->N
\draw[decorate,decoration={brace,mirror,amplitude=4pt}]
  ([yshift=-2pt]bundle.south west) -- ([yshift=-2pt]bundle.south east);

\coordinate (rightbrace) at
  ([yshift=-8pt]$(bundle.south west)!0.5!(bundle.south east)$);

% punto de bifurcación
\coordinate (split) at
  ([yshift=-18pt]$(leftbrace)!0.58!(rightbrace)$);

% tronco común
\draw[<-]
  (leftbrace)
  .. controls +(0,-9pt) and +(-15pt,0) ..
  (split);

% rama hacia la llave bajo p:\mathcal P\to N
\draw[->]
  (split)
  .. controls +(15pt,0) and +(0,-9pt) ..
  (rightbrace);

% rama hacia la forma de soldadura
\draw[->]
  (split)
  .. controls +(24pt,-1pt) and +(0,-16pt) ..
  (theta.south);

% correspondencia nabla <-> gamma
\draw[<->]
  (nabla.north) to[out=25,in=155] (gamma.north);
\end{tikzpicture}
\]
\smallskip

\begin{remark}\label{rmk:recognition-correspondence-euc}
{\rm We record when a Cartan geometry $(p:\mathcal P\to N,\omega)$ of type $(G,H)$ arises, locally, as
the unit tangent bundle of a Riemannian manifold; that is, as a correspondence space of a Cartan
geometry of type $(G,P)$. Under the identification $\mathfrak m\cong\mathfrak g/\mathfrak h$, the
subspace $B$ lies in $\mathfrak p$ and $\mathfrak p=\mathfrak h\oplus B$, so $B$ is a set of
representatives for $\mathfrak p/\mathfrak h$, and the distribution
$\mathcal V(N)=\mathcal P\times_H(\mathfrak p/\mathfrak h)$ of \cite[Theorem 1.5.14]{CS09} is exactly
the distribution $\mathcal B$ induced by $\omega$. The criterion therefore specializes as follows:
$(p:\mathcal P\to N,\omega)$ is locally a correspondence space of a Cartan geometry of type $(G,P)$
if and only if $\mathcal B$ is integrable and its curvature function satisfies
\begin{equation}\label{eq:recognition-correspondence-curvature}
\kappa(u)(X,Y)=0,\qquad u\in\mathcal P,\ X\in B,\ Y\in\mathfrak m,
\end{equation}
equivalently $K\big(\omega^{-1}(X),\omega^{-1}(Y)\big)=0$. Indeed, since $\mathfrak p= \mathfrak h\oplus B$,
the general condition $\kappa(u)(X+\mathfrak h,\cdot)=0$, $X\in\mathfrak p$, is precisely
\eqref{eq:recognition-correspondence-curvature}. When it holds, for any sufficiently small local
twistor space $f:U\to M$, there exists a Cartan geometry of type $(G,P)$ on $M$ such that the
restricted Cartan geometry $(p:p^{-1}(U)\to U,\left.\omega\right|_{U})$ is isomorphic to an open subset of the
correspondence space $\mathcal C(M)$; and since $P/H\cong\mathbb S^{n-1}$ is connected for $n\ge2$,
this Cartan geometry on $M$ is unique.}
\end{remark}

\subsection{The normal Cartan connection}
\label{subsec:normal-adapted-cartan}
We now single out a canonical Cartan connection associated with the underlying
structure \((\bar g,U,\alpha,\mathcal E,\mathcal B,\Psi)\). By the construction given in the proof of
Theorem~\ref{thm:euc-onminusone-equivalence}, the data
\((\bar g,U,\alpha,\mathcal E,\mathcal B,\Psi)\) determine the corresponding
\(H\)-principal fiber bundle and its soldering form \(\vartheta\). Thus, for
\(n\geq3\), the remaining freedom in choosing a Cartan geometry of type \((G,H)\)
inducing this underlying structure lies entirely in the principal-connection part
\(\gamma\) of the Cartan connection \(\omega=\gamma+\vartheta\). We remove this freedom
by a normalization procedure adapted from the standard treatment of first-order structures with structure group \(H\) and trivial first prolongation to the present setting; see \cite[Section~1.6.1]{CS09} and
\cite[Section~2.2]{CS17}. In the present formulation, the normalization condition is
imposed directly on the \(\mathfrak m\)-part of the Cartan curvature, that is, on the
torsion of the Cartan connection.

On $\mathfrak h$ we fix an $H$-invariant Euclidean inner product $\langle\cdot,\cdot\rangle_{\mathfrak h}$, which exists since $H$ is compact. We now define
\[
\partial:\mathfrak m^*\otimes\mathfrak h
\longrightarrow
\Lambda^2\mathfrak m^*\otimes\mathfrak m
\]
by
\begin{equation}\label{eq:normal-spencer}
(\partial A)(X,Y)
=
[A(X),Y]
-
[A(Y),X],
\qquad
X,Y\in\mathfrak m.
\end{equation}
Since $\mathfrak m$ is $H$-invariant we have
$[\mathfrak h,\mathfrak m]\subset\mathfrak m$, so the brackets in \eqref{eq:normal-spencer} automatically lie in $\mathfrak m$; moreover $(\partial A)(X,Y)$ is antisymmetric in $X,Y$ by construction, so
indeed $\partial A\in\Lambda^2\mathfrak m^*\otimes\mathfrak m$. This operator is known as the Spencer differential. Let
\[
\partial^*\colon
\Lambda^2\mathfrak m^*\otimes\mathfrak m
\longrightarrow
\mathfrak m^*\otimes\mathfrak h
\]
be the adjoint of \(\partial\), characterised by
\begin{equation}\label{eq:normal-adjoint-definition}
\langle \partial A,T\rangle_{\Lambda^2\mathfrak m^*\otimes\mathfrak m}
=
\langle A,\partial^*T\rangle_{\mathfrak m^*\otimes\mathfrak h},
\qquad
A\in\mathfrak m^*\otimes\mathfrak h,\quad T\in\Lambda^2\mathfrak m^*\otimes\mathfrak m.
\end{equation}
The two inner products here are the canonical ones induced by
$\langle\,\cdot\,,\cdot\,\rangle_{\mathfrak m}$ and by
$\langle\,\cdot\,,\cdot\,\rangle_{\mathfrak h}$. Since both spaces are finite-dimensional, $\partial^{*}$ is well defined and unique; and, by
non-degeneracy of the inner product, $\mathrm{Ker}\,(\partial^{*})=(\operatorname{Im}(\partial))^{\perp}$. The injectivity of $\partial$ follows from the general discussion in \cite[Section~1.6.1]{CS09}; we give a direct proof here for completeness.

\begin{lemma}
\label{lem:normal-partial-injective}
For \(n\geq3\), the map $\partial:\mathfrak m^*\otimes\mathfrak h
\longrightarrow
\Lambda^2\mathfrak m^*\otimes\mathfrak m$ is injective.
\end{lemma}

\begin{proof}
Let
\[
A:\mathfrak m\longrightarrow\mathfrak h
\]
satisfy \(\partial A=0\). Taking $Y=e_0$ in \eqref{eq:normal-spencer}, and using $[\mathfrak h,e_0]=0$, gives
\[
0=(\partial A)(X,e_0)=-[A(e_0),X]
\]
for every $X\in\mathfrak m$. Note that if $C\in\mathfrak h$, then its bracket action on
$\mathfrak m$ is given by $[C,X]=(0,Cv,Cw)$ for $X=(v_0,v,w)\in\mathfrak m$. In particular, $C$ acts on the $E$- and $B$-components by the standard
representation of $\mathfrak h$. Since this representation is faithful, the condition
$[A(e_0),X]=0$ for all $X\in\mathfrak m$ forces $A(e_0)=0$. Now take $X\in E$ and $Y\in B$. Then
\[
0=(\partial A)(X,Y)=[A(X),Y]-[A(Y),X].
\]
The first term lies in $B$ and the second in $E$, so both vanish since
$\mathfrak m=\mathbb Re_0\oplus E\oplus B$ is a direct sum. Letting $Y$ range over $B$, faithfulness
gives $A(X)=0$ for every $X\in E$; the same argument with the roles of $E$ and $B$ interchanged
gives $A(Y)=0$ for every $Y\in B$. Together with $A(e_0)=0$, this gives $A=0$.
\end{proof}

\begin{remark}\label{rmk:first-prolongation}
{\rm In the terminology of \cite[Section~1.6.1]{CS09}, the kernel of $\partial$ is called the first prolongation of $\mathfrak h$ and denoted $\mathfrak h^{(1)}$. By
Lemma~\ref{lem:normal-partial-injective}, it is trivial for $n\geq3$.}
\end{remark}

\begin{definition}
\label{def:normal-adapted-cartan}
Let $(p:\mathcal P\to N,\omega)$ be a Cartan geometry of type $(G,H)$, and let $K^\omega=d\omega+\frac12[\omega,\omega]$ be its curvature form.  The torsion of $\omega$ is the $\mathfrak m$-part of its curvature function
$\kappa^\omega$, that is, the map
\[
T^\omega:=\operatorname{pr}_{\mathfrak m}\circ\,\kappa^\omega\colon
\mathcal P\longrightarrow\Lambda^2\mathfrak m^*\otimes\mathfrak m,
\qquad
T^\omega(u)(X,Y)=\operatorname{pr}_{\mathfrak m}\left[K^\omega(u)\big(\omega^{-1}(X)(u),\omega^{-1}(Y)(u)\big)\right].
\]
We say that $\omega$ is normal if $\partial^*T^\omega=0$.
\end{definition}

\begin{remark}\label{rmk:equivariance}
{\rm For later use, we record the $H$-equivariance of $\partial$, $\partial^{*}$ and $\partial^{*}T^{\omega}$. Two
notions of equivariance are involved. A linear map $\phi\colon V\to W$ between $H$-modules is called
$H$-equivariant if $\phi(h\cdot v)=h\cdot\phi(v)$ for all $h\in H$ and $v\in V$, equivalently
$\phi\circ(h\cdot)=(h\cdot)\circ\phi$; this is the sense in which $\partial$ and $\partial^{*}$ are
equivariant. A function $f\colon\mathcal P\to\mathbb V$ with values in an $H$-module $\mathbb V$ is called
$H$-equivariant if $f(uh)=h^{-1}\cdot f(u)$ for all $u\in\mathcal P$ and $h\in H$, and such functions are
precisely the sections of $\mathcal P\times_H\mathbb V$; this is the sense in which $\partial^{*}T^{\omega}$
is equivariant.

On the tensor spaces \(\mathfrak m^{*}\otimes\mathfrak h\) and \(\Lambda^{2}\mathfrak m^{*}\otimes\mathfrak m\)
we consider the natural actions induced by $\underline{\operatorname{Ad}}$ on $\mathfrak m$ and
$\operatorname{Ad}$ on $\mathfrak h$, which we now make explicit. On $\mathfrak m^{*}$ we take the dual
\(H\)-module structure,
\[
(h\cdot\xi)(X)=\xi\big(\underline{\operatorname{Ad}}(h^{-1})(X)\big),\qquad \xi\in\mathfrak m^{*},\ X\in\mathfrak m,
\] while on a tensor product of $H$-modules, $H$ acts by
$h\cdot(v\otimes w)=(h\cdot v)\otimes(h\cdot w)$. By linearity, it is enough to compute the action on tensors of the form
\(A=\xi\otimes Z\), with \(\xi\in\mathfrak m^{*}\) and \(Z\in\mathfrak h\), giving
\begin{equation}\label{eq:induced-action-A}
(h\cdot A)(X)
=
\xi\big(\underline{\operatorname{Ad}}(h^{-1})(X)\big)\operatorname{Ad}(h)(Z)
=
\operatorname{Ad}(h)\big(
A(\underline{\operatorname{Ad}}(h^{-1})(X))
\big).
\end{equation}
The same argument applied to
\(\Lambda^{2}\mathfrak m^{*}\otimes\mathfrak m\) gives
\begin{equation}\label{eq:induced-action-T}
(h\cdot T)(X,Y)
=
\underline{\operatorname{Ad}}(h)\big(
T(
\underline{\operatorname{Ad}}(h^{-1})(X),
\underline{\operatorname{Ad}}(h^{-1})(Y)
)
\big).
\end{equation}
These formulas hold for all
\(A\in\mathfrak m^{*}\otimes\mathfrak h\),
\(T\in\Lambda^{2}\mathfrak m^{*}\otimes\mathfrak m\), and
\(X,Y\in\mathfrak m\).

Since $\langle\cdot,\cdot\rangle_{\mathfrak m}$ and $\langle\cdot,\cdot\rangle_{\mathfrak h}$ are
$H$-invariant, $H$ acts by isometries on $\mathfrak m^{*}\otimes\mathfrak h$ and $\Lambda^{2}\mathfrak m^{*}\otimes\mathfrak m$, with the actions \eqref{eq:induced-action-A} and
\eqref{eq:induced-action-T} and the inner products of \eqref{eq:normal-adjoint-definition}. An invertible operator on a finite-dimensional inner
product space is an isometry if and only if its adjoint equals its inverse; applied to $h\cdot$, invertible with inverse $h^{-1}\cdot$, this yields $(h\cdot)^{*}=h^{-1}\cdot$ on each of the two tensor spaces.

We now check the equivariance of $\partial$ and $\partial^{*}$. For $\partial$, using the definition of $\partial$ and \eqref{eq:induced-action-A}, we obtain, for $X,Y\in\mathfrak m$,
\[
\begin{aligned}
\big(\partial(h\cdot A)\big)(X,Y)
&=\big[(h\cdot A)(X),\,Y\big]-\big[(h\cdot A)(Y),\,X\big]\\
&=\Big[\operatorname{Ad}(h)\big(A(\underline{\operatorname{Ad}}(h^{-1})(X))\big),\,Y\Big]
 -\Big[\operatorname{Ad}(h)\big(A(\underline{\operatorname{Ad}}(h^{-1})(Y))\big),\,X\Big]\\
&=\underline{\operatorname{Ad}}(h)\Big((\partial A)\big(\underline{\operatorname{Ad}}(h^{-1})(X),\underline{\operatorname{Ad}}(h^{-1})(Y)\big)\Big)
 =\big(h\cdot(\partial A)\big)(X,Y),
\end{aligned}
\]
the last equality by \eqref{eq:induced-action-T} applied to
$\partial A\in\Lambda^{2}\mathfrak m^{*}\otimes\mathfrak m$; thus
$\partial\circ(h\cdot)=(h\cdot)\circ\partial$. Passing to $\partial^{*}$, we take adjoints in this relation, which reverses the order to
$(h\cdot)^{*}\circ\partial^{*}=\partial^{*}\circ(h\cdot)^{*}$; substituting the identity
$(h\cdot)^{*}=h^{-1}\cdot$ obtained above gives $(h^{-1}\cdot)\circ\partial^{*}=\partial^{*}\circ(h^{-1}\cdot)$,
which, as $h$ ranges over $H$, reads $\partial^{*}(h\cdot T)=h\cdot(\partial^{*}T)$.

Finally, the torsion $T^{\omega}$ is $H$-equivariant (see \cite[Section~1.5]{CS09}), i.e., $T^{\omega}(uh)=h^{-1}\cdot T^{\omega}(u)$ for all $u\in\mathcal P$ and $h\in H$. Combining the equivariance of $\partial^{*}$ with that of $T^{\omega}$, 
we have
\[
(\partial^{*}T^{\omega})(uh)=\partial^{*}\big(T^{\omega}(uh)\big)=\partial^{*}\big(h^{-1}\cdot T^{\omega}(u)\big)
=h^{-1}\cdot\big(\partial^{*}T^{\omega}\big)(u),
\]
so $\partial^{*}T^{\omega}\colon\mathcal P\to\mathfrak m^{*}\otimes\mathfrak h$ is $H$-equivariant in the
second sense, $(\partial^{*}T^{\omega})(uh)=h^{-1}\cdot(\partial^{*}T^{\omega})(u)$; equivalently, it is a
section of $\mathcal P\times_H(\mathfrak m^{*}\otimes\mathfrak h)$.}
\end{remark}

\begin{theorem}\label{thm:normal-adapted-connection}
Assume \(n\geq3\). Among all Cartan geometries of type \((G,H)\) on a manifold $N$ whose Cartan connections share the same \(\mathfrak m\)-part \(\vartheta\), there exists a unique one whose Cartan connection is normal. Equivalently, the condition \(\partial^*T^\omega=0\) determines the principal-connection part uniquely.
\end{theorem}

\begin{proof}
Fix an $H$-principal connection $\gamma^0$ on $\mathcal P$ and set $\omega^0:=\gamma^0+\vartheta$ and
$T^0:=T^{\omega^0}$. Given two $H$-principal connections on $\mathcal P$, they differ by a horizontal,
$H$-equivariant $\mathfrak h$-valued $1$-form; conversely, the sum of such a $1$-form and an $H$-principal
connection is again an $H$-principal connection. Hence every Cartan connection with $\mathfrak m$-part
$\vartheta$ can be written, in terms of $\gamma^0$, as
\[
\omega^a=(\gamma^0+a)+\vartheta,
\]
with $a$ a horizontal, $H$-equivariant $\mathfrak h$-valued $1$-form on $\mathcal P$. Being horizontal, $a$
factors through the soldering form: for each $u\in\mathcal P$ the linear map
$a(u)\colon T_u\mathcal P\to\mathfrak h$ vanishes on the vertical subspace $\ker\vartheta(u)$, hence factors
uniquely through the surjection $\vartheta(u)\colon T_u\mathcal P\to\mathfrak m$, yielding a unique
$A(u)\in\mathfrak m^{*}\otimes\mathfrak h$ with $a(u)=A(u)\circ\vartheta(u)$, that is,
$a(u)(\xi)=A(u)\big(\vartheta(u)(\xi)\big)$ for $\xi\in T_u\mathcal P$. This defines a map
$A\colon\mathcal P\to\mathfrak m^{*}\otimes\mathfrak h$, which is $H$-equivariant in the sense of
Remark~\ref{rmk:equivariance}, $A(uh)=h^{-1}\cdot A(u)$.

We first record how the torsion changes with $a$, claiming that pointwise
\begin{equation}\label{eq:normal-torsion-change}
T^{\omega^a}(u)=T^0(u)+\partial\big(A(u)\big).
\end{equation}
Since $\omega^a=\omega^0+a$, the curvature forms satisfy
\[
K^{\omega^a}=K^{\omega^0}+da+[\omega^0,a]+\tfrac12[a,a].
\]
Fix $X,Y\in\mathfrak m$ and put $\xi_X:=(\omega^0)^{-1}(X)$ and $\xi_Y:=(\omega^0)^{-1}(Y)$. For
$X\in\mathfrak m$ the constant vector fields of $\omega^a$ and of $\omega^0$ differ by a vertical field,
$(\omega^a)^{-1}(X)=\xi_X-\zeta_{A(\cdot)(X)}$; since the curvature form of a Cartan connection is horizontal, evaluating
$K^{\omega^a}$ on the $\omega^a$-constant fields agrees with evaluating it on $\xi_X,\xi_Y$, so
\[
T^{\omega^a}(u)(X,Y)=\operatorname{pr}_{\mathfrak m}\big[K^{\omega^a}(u)\big(\xi_X(u),\xi_Y(u)\big)\big].
\]
We evaluate the four terms. First,
$\operatorname{pr}_{\mathfrak m}\big[K^{\omega^0}(u)\big(\xi_X(u),\xi_Y(u)\big)\big]=T^0(u)(X,Y)$.
The forms $da$ and $[a,a]$ are $\mathfrak h$-valued, hence annihilated by $\operatorname{pr}_{\mathfrak m}$.
Finally, using $\omega^0(u)(\xi_X(u))=X$, $\omega^0(u)(\xi_Y(u))=Y$ and $a(u)(\xi_X(u))=A(u)(X)$,
$a(u)(\xi_Y(u))=A(u)(Y)$,
\[
[\omega^0,a](u)\big(\xi_X(u),\xi_Y(u)\big)=[X,A(u)(Y)]-[Y,A(u)(X)],
\]
and both terms already lie in $\mathfrak m$, so $\operatorname{pr}_{\mathfrak m}$ acts as the identity and, by
antisymmetry of the bracket,
\[
\operatorname{pr}_{\mathfrak m}\big[[\omega^0,a](u)\big(\xi_X(u),\xi_Y(u)\big)\big]
=[A(u)(X),Y]-[A(u)(Y),X]=\big(\partial\big(A(u)\big)\big)(X,Y).
\]
Adding the four contributions gives \eqref{eq:normal-torsion-change}.

By \eqref{eq:normal-torsion-change}, $\omega^a$ is normal if and only if
\begin{equation}\label{eq:normalising-equation}
0=\partial^{*}T^{\omega^a}=\partial^{*}T^0+\partial^{*}\partial A.
\end{equation}
Since $\partial$ is injective (Lemma~\ref{lem:normal-partial-injective}) and the inner products are positive
definite,
\[
\langle\partial^{*}\partial A,A\rangle_{\mathfrak m^{*}\otimes\mathfrak h}
=\langle\partial A,\partial A\rangle_{\Lambda^{2}\mathfrak m^{*}\otimes\mathfrak m}
>0\quad\text{for }A\neq0.
\]
Hence the self-adjoint operator $\partial^{*}\partial$ on the finite-dimensional space
$\mathfrak m^{*}\otimes\mathfrak h$ is positive definite, in particular invertible, so
\eqref{eq:normalising-equation} has the unique solution
\begin{equation}\label{eq:normal-cartan-correction}
A=-(\partial^{*}\partial)^{-1}\partial^{*}T^0.
\end{equation}
By Remark~\ref{rmk:equivariance}, $\partial$, $\partial^{*}$ and $\partial^{*}T^0$ are $H$-equivariant; so is
$(\partial^{*}\partial)^{-1}$, being the inverse of the $H$-equivariant isomorphism $\partial^{*}\partial$.
Hence the solution \eqref{eq:normal-cartan-correction} is $H$-equivariant and defines a Cartan connection,
which by the above is the unique normal one with $\mathfrak m$-part $\vartheta$.
\end{proof}
In view of Theorem~\ref{thm:euc-onminusone-equivalence}, the normalisation
established in Theorem~\ref{thm:normal-adapted-connection} can be summarised
by the following diagram, where the double arrow denotes a one-to-one
correspondence:

\smallskip

\[
\boxed{
\begin{gathered}
(\bar g,U,\alpha,\mathcal E,\mathcal B,\Psi)
\text{ satisfying}\\[-1mm]
\text{properties }
\textit{\ref{it:metric}}-\textit{\ref{it:psi}}\text{ of}\\[-1mm]
\text{Proposition~\ref{thm:euc-onminusone-structure-new}}
\end{gathered}
}
\quad\Longleftrightarrow\quad
\boxed{
\begin{gathered}
\text{Cartan geometries }(p:\mathcal P\to N,\omega)\text{ of}\\[-1mm]
\text{type }(G,H)
\text{ such that }\partial^{*}T^{\omega}=0
\end{gathered}
}
\]
\smallskip

\begin{remark}
{\rm The normalisation carried out in this section can be performed whenever $G/H$ is a 
reductive Klein geometry, $\mathfrak h$ has trivial first prolongation, and $\mathfrak h$ and $\mathfrak m$
carry $H$-invariant inner products. The normalisation does not depend on the $H$-invariant inner product chosen on $\mathfrak h$, as is easily verified.}
\end{remark}
We take $\partial^{*}T^{\omega}=0$ to be the natural normalisation for Cartan geometries of type $(G,H)$:
the Cartan connection of a Riemannian geometry, regarded as one of type $(G,H)$ on its correspondence space, is normal. It is in this sense that the normalisation is natural, and the following proposition makes this precise.

\begin{proposition}
\label{prop:normal-compatible-with-riemannian-correspondence}
Assume \(n\geq3\). Let $(M,g)$ be a Riemannian manifold, and let $(q:\mathcal F_g(M)\to M,\omega)$ be the Riemannian geometry
associated with $(M,g)$. Under the identification $\mathcal U M\cong\mathcal F_g(M)/H$, the Cartan connection
$\omega$, regarded as one of type $(G,H)$ on $p:\mathcal F_g(M)\to\mathcal U M$, is the normal Cartan connection whose underlying structure is $(\widehat g,\mathbf Z_g,\tau,\mathcal E,V,\Phi)$.
\end{proposition}

\begin{proof}
Recall that
\[
\mathfrak p=\mathfrak h\oplus B,
\qquad
\mathfrak m=\mathfrak t\oplus B
=\mathbb Re_0\oplus E\oplus B.
\]
By Proposition~\ref{prop:unit-bundle-structures-agree-new}, the underlying structure induced by $\omega$ is
exactly $(\widehat g,\mathbf Z_g,\tau,\mathcal E,V,\Phi)$. It remains only to verify normality.

Fix $u\in\mathcal F_g(M)$. Since $\omega$ is torsion-free as a Cartan connection of type $(G,P)$, the curvature form takes values in
$\mathfrak p$. Using the decomposition $\mathfrak p=\mathfrak h\oplus B$, together with the fact that
$\operatorname{pr}_{\mathfrak m}$ vanishes on $\mathfrak h$ and restricts to the identity on
$B\subset\mathfrak m$, we obtain
\[
T^\omega(u)\in\Lambda^2\mathfrak m^*\otimes B.
\]
Moreover, as the curvature form is horizontal with respect to
$q:\mathcal F_g(M)\to M$ and $B\subset\mathfrak p$, we have
$T^\omega(u)(X,\cdot)=0$ for every $X\in B$. Consequently the only pairs on which $T^\omega(u)$ can be non-zero are those with both arguments in
$\mathbb Re_0\oplus E$, and there $T^\omega(u)$ takes values in $B$.

Now let $A\in\mathfrak m^*\otimes\mathfrak h$ and compute
$\langle \partial A,T^\omega(u)\rangle_{\Lambda^2\mathfrak m^*\otimes\mathfrak m}$ term by term in the
orthonormal basis $(e_0,e_1,\ldots,e_{n-1},b_1,\ldots,b_{n-1})$ of $\mathfrak m$. If both basis vectors lie
in $\mathbb Re_0\oplus E$, then $(\partial A)(X,Y)=[A(X),Y]-[A(Y),X]\in E$, because $[\mathfrak h,e_0]=0$ and
$[\mathfrak h,E]\subseteq E$; as $T^\omega(u)(X,Y)\in B$ and $(\mathbb Re_0\oplus E)\perp B$, the
corresponding term vanishes. If at least one basis vector lies in $B$, then $T^\omega(u)$ vanishes on that
pair, so the term vanishes as well. Hence
\[
\langle \partial A,T^\omega(u)\rangle_{\Lambda^2\mathfrak m^*\otimes\mathfrak m}=0,\text{ for every }A\in\mathfrak m^*\otimes\mathfrak h,
\]
which is $(\partial^*T^\omega)(u)=0$ by \eqref{eq:normal-adjoint-definition}. Since $u$ was arbitrary,
$\partial^*T^\omega=0$.
\end{proof}
 
\begin{remark}\label{rmk:future-work}
{\rm For $n\geq 3$, once the $\mathfrak m$-part $\vartheta$ is fixed, Theorem~\ref{thm:normal-adapted-connection} uniquely determines the principal-connection part $\gamma$ for which the Cartan connection $\omega=\gamma+\vartheta$ is normal. The resulting $H$-principal connection $\gamma$ induces a distinguished linear connection $\nabla$ on $N$, namely the linear-connection component of the associated Cartan datum. In future work, we plan to compute $\nabla$ explicitly and to study its geometric properties, including its torsion, curvature, and geodesics.}
\end{remark}

\begin{remark}\label{rmk:nilpotent-not-correspondence}
{\rm
A natural source of Cartan geometries of type $(G,H)$ is that provided by the
correspondence space construction: through it, every unit tangent bundle carries a
canonical Cartan geometry of type $(G,H)$ with normal Cartan connection
(Proposition~\ref{prop:normal-compatible-with-riemannian-correspondence}). This
construction, however, does not yield all geometries of type $(G,H)$, not even
locally: there exist Cartan geometries of type $(G,H)$ that are not locally isomorphic
to the correspondence space of any Cartan geometry of type $(G,P)$. We construct such
a family on certain nilpotent Lie groups.

Fix $n\ge 3$ and a non-zero skew-symmetric $S=(s_{ij})\in\mathfrak h$. On $\R^{2n-1}$,
with distinguished basis $(e_0,e_1,\dots,e_{n-1},b_1,\dots,b_{n-1})$, define a bracket
by
\[
[b_i,b_j]=s_{ij}e_0\qquad(1\le i,j\le n-1),
\]
and all remaining brackets among basis vectors vanish. Every bracket takes values in
the central line $\R e_0$; the Jacobi identity therefore holds automatically, and
$\mathfrak n(S):=(\R^{2n-1},[\cdot,\cdot])$ is a two-step nilpotent Lie algebra. Let $N(S)$
be the associated simply connected Lie group, and for
$X\in\mathfrak n(S)$ let $L_X$ denote the corresponding left-invariant vector field on
$N(S)$.

Declaring the left-invariant frame
$(L_{e_0},L_{e_1},\dots,L_{e_{n-1}},L_{b_1},\dots,L_{b_{n-1}})$ to be orthonormal
defines a left-invariant Riemannian metric $\bar g$ on $N(S)$. Set
\[
U:=L_{e_0},\qquad \alpha:=\bar g(U,\cdot),\qquad
\mathcal E:=\operatorname{span}\{L_{e_1},\dots,L_{e_{n-1}}\},\qquad
\mathcal B:=\operatorname{span}\{L_{b_1},\dots,L_{b_{n-1}}\},
\]
and let $\Psi\in\mathcal T_{(1,1)}(N(S))$ be the left-invariant $(1,1)$-tensor field
determined by
\[
\Psi(U)=0,\qquad \Psi(L_{e_i})=L_{b_i},\qquad
\Psi(L_{b_i})=-L_{e_i}\qquad(1\le i\le n-1).
\]
All these tensors have constant components in the frame above, so a direct
computation on the basis shows that $(\bar g,U,\alpha,\mathcal E,\mathcal B,\Psi)$
satisfies properties \textit{\ref{it:metric}}--\textit{\ref{it:psi}} of
Proposition~\ref{thm:euc-onminusone-structure-new}; in particular $TN(S)=\R
U\oplus\mathcal E\oplus\mathcal B$ is an orthogonal splitting with
$\operatorname{rank}\mathcal E=\operatorname{rank}\mathcal B=n-1$, and $(\bar
g,U,\alpha,\Psi)$ is an almost contact metric structure. By
Theorems~\ref{thm:euc-onminusone-equivalence}
and~\ref{thm:normal-adapted-connection}, this structure determines a unique Cartan
geometry of type $(G,H)$ on $N(S)$ whose Cartan connection is normal.

Since $S\neq0$ and $n\ge3$, there exist indices $i\neq j$ with $s_{ij}\neq0$, and then
\[
[L_{b_i},L_{b_j}]=s_{ij}U\notin\Gamma(\mathcal B),
\]
so $\Gamma(\mathcal B)$ is not closed under the Lie
bracket; by the Frobenius theorem, $\mathcal B$ is therefore not
integrable. Remark~\ref{rmk:recognition-correspondence-euc} shows that the Cartan
geometry on $N(S)$ can be locally a correspondence space of a Cartan geometry of type
$(G,P)$ only if $\mathcal B$ is integrable; hence, this never occurs.

This construction generalizes naturally to a much wider family of examples.
}
\end{remark}

\end{document}